\documentclass{article}
\usepackage{amsthm}
\newtheorem{theorem}{Theorem}[section]

\newtheorem{lemma}[theorem]{Lemma}

\newtheorem{cor}[theorem]{Corollary}
{\theoremstyle{definition}
\newtheorem{defin}[theorem]{Definition}

\newtheorem{rem}[theorem]{Remark}

{\theoremstyle{definition}
\newtheorem{exam}[theorem]{Example}
\newcommand{\mycomment}[1]{}

\usepackage{amsmath,amssymb, amsfonts}
\usepackage[margin=1.5in]{geometry}
\usepackage{mathrsfs}
\usepackage{enumerate}
\usepackage{graphicx}
\graphicspath{ {./images/} }
\usepackage{nicefrac}
\usepackage{setspace}
\usepackage{relsize}
\usepackage{hyperref}
\usepackage[numbers,sort&compress]{natbib}

\DeclareMathOperator{\interior}{int}

\DeclareMathOperator{\ess}{ess}

\title{Symbolic dynamics, shadowing and representation of real numbers with some countably piecewise linear Markov maps}
\author{Zoltán Kalocsai \thanks{Supported by the Eköp-24 University Excellence Scholarship Program of the Ministry for Culture and Innovation from the source of the National Research, Development and Innovation Fund. Supported by ELTE Eötvös Loránd University, Budapest, Hungary Faculty of Science}}
\date{ }

\begin{document}

\maketitle

\begin{abstract}
We study piecewise linear Markov maps, with countable Markov partitions, inspired by a problem of the Miklós Schweitzer competition of the János Bolyai Mathematical Society in 2022. We introduce $\ell$-Markov partitions and apply ideas of symbolic dynamics to our systems, relating them to Markov shifts. We prove the shadowing property for the system from the competition. We also investigate the possible orbits of rational numbers, for a class of piecewise linear Markov maps which generalize our original system. This has connections with symbolic dynamics, Diophantine approximations and Cantor series. We prove statements on the eventual periodicity of rationals and provide some example systems with different properties.
\end{abstract}

\tableofcontents

\renewcommand{\thefootnote}{}

\section{Introduction}

Piecewise linear interval maps are one of the simplest discrete time dynamical systems. Thanks to their simple definition they are widely used for modelling phenomena, and by their research we can learn techniques widely applicable for more complex systems.\medskip

In this work we consider dynamical properties of a class of real-to-real maps which have a countable Markov partition and are linear on the elements of the partition. These we call systems with an $\ell$-Markov partition. Our analysis of these maps is further motivated by a problem of the 2022. Miklós Schweitzer competition, proposed by my advisor Zoltán Buczolich. The Miklós Schweitzer competition is organized by the János Bolyai Mathematical Society. Its website can be found at \cite{schweitz}. We state the mentioned problem.\medskip

\textit{
The transformation $T:\: [ 0,\infty ) \rightarrow [0, \infty)$ is linear on every positive integer interval, and its values at integers are defined as follows:
\begin{equation}
T(n)=
	\begin{cases}
		0 & \text{if } 2|n\\
		4^{\ell}+1 & \text{if } 2 \nmid n,\, 4^{\ell - 1}\leq n < 4^{\ell} \: (\ell \in \mathbb{Z}^+). 
	\end{cases}
\end{equation}
Let $T^0(x)=x$ and $T^{n}(x) = T(T^{n-1}(x))$ for all $n>0$. Find the $\liminf_{k\rightarrow \infty}(T^k(x))$ and the $\limsup_{k\rightarrow \infty}(T^k(x))$ for Lebesgue almost every $x \in [0, \infty)$.\medskip
}

\begin{figure}[h]
\includegraphics[width= 10cm]{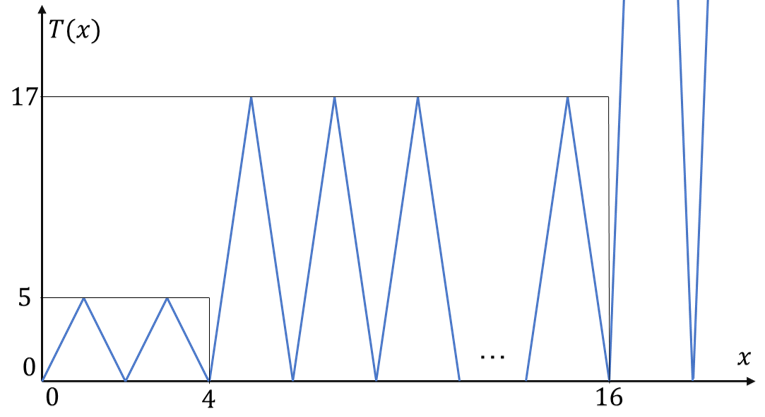}
\centering
\caption{A figure of the transformation $T$.}
\label{Fuggveny}
\end{figure}

Methods applicable for the solution of this problem will be discussed in another paper of the author in preparation \cite{mineOfficial}. In that work we discuss the possible invariant measures systems with $\ell$-Markov partitions can have, the asymptotic behaviour of pushforward measures and other ergodic properties. Here we only use the system presented.\medskip

In Section \ref{sec1} we first build up the relevant theory of symbolic dynamics in our special case. Then we introduce the notion of expansivity, which is a weaker form of the 'expanding condition' appearing widely in literature, for example in \cite{Bugiel1998}. We show the topological conjugacy of the studied systems with shifts if expansivity is satisfied.\medskip

In Section \ref{shadow} we introduce the pseudo-orbit shadowing property, based on the definitions from the literature. The phenomena presented in this section aim to model the use of imperfect numerical simulations for investigating the orbits of chaotic dynamical systems. We prove the shadowing property for the original system from the Schweitzer problem. The original system is in many ways similar to tent maps. Their shadowing property is studied in \cite{Tent}. Tent maps have a compact domain in contrast to our system which is non-compact. The classical shadowing theorem appearing in \cite{ShadowBook} as well investigates shadowing near a compact invariant set. It also assumes the mapping to be a diffeomorphism, which again does not apply to our system.\medskip

We continue in Section \ref{orbit of rationals} by investigating what symbolic trajectories can represent rational numbers. This section has connections with the rationality of a Cantor series. For a general review on when do Cantor series represent rational numbers see \cite{CantorIntro}. We consider piecewise linear Markov maps with countably many segments, where the intervals of the type $[n-1,n], n \in \mathbb{Z}$ form a Markov partition, and where the slope is at least 2 on any interval. We do not consider points which have an eventually integer orbit. We state some general results on the relationship of the set of eventually periodic points and the set of rationals. We achieve these results by employing methods of Diophantine approximation. We gain insight into how real numbers can be represented by symbolic orbits of a Markov shift over a countable alphabet.\medskip

\begin{defin}[Cantor series]
For some positive integer sequence $\{q_i\}_{i=0}^\infty$, and some integer sequence $0\leq  a_i \leq q_i - 1$ the following is called a Cantor series:
$$\sum_{n=0}^{\infty} \frac{a_n}{\prod_{k=1}^n q_k}.$$
\end{defin}

In Section \ref{theoretical_results} we show in Theorem \ref{bondolt} that in the systems considered, rationals with bounded orbits are eventually periodic. In Corollary \ref{Ergod}, we provide bounds on the speed of divergence to infinity of the average magnitude of rational orbits that are not eventually periodic. The main proposition of this section is Theorem \ref{last}, which gives a condition determining a class of systems for which all rational points are eventually periodic. The proof of the statement depends on the rationality of Cantor series. We describe the connections with \cite{Oppenheim}.\medskip

In Section \ref{example_systems}, some examples are given to demonstrate how the periodicity of rational points relates to the stochastic properties of the associated Markov chain. The main takeaway is that there is no direct link between the two. Example \ref{kek} shows the differences between representation by a symbolic trajectory and by a Cantor series, and thus justifies some further assumptions on the systems considered. Example \ref{zold} and Example \ref{extransi} are systems with a transient Markov chain for which Theorem \ref{last} applies, thus all rationals are eventually periodic. Example \ref{posrecEscapes} is a positively recurrent system with rational orbits that are not eventually periodic.\medskip

\section{Preliminary notation}

\begin{defin}
The function $\ell:\: \mathbb{R}_{\geq 0} \rightarrow \mathbb{N}$ is defined as follows:\[
\ell(x) = \inf\{n\in \mathbb{Z}^+:\: x \leq 4^n \}.
\]
\end{defin}

For a finite or infinite sequence $\underline{i} = (i_0, i_1, ...)$ we denote by $\underline{i}|n$ the finite sequence $(i_0, ..., i_n)$. For a Lebesgue-measurable subset $X$ of $\mathbb{R}^n$ we will denote the Lebesgue $\sigma$-algebra on $X$ with $\mathscr{L}(X)$. If a function $f$ is linear on the interior of an interval $H \subset \mathbb{R}$, then by $f'(H)$ we mean $f'(x)$ for any $x\in \interior (H)$. For some sequence of reals $x_n$ and for some integers $a>b$ we define $\sum_{n=a}^b x_n = 0$ and $\prod_{n=a}^b x_n = 1$ as it is standard for the empty sum and product. Throughout this paper by almost every we mean Lebesgue almost every, and by measurable we mean Lebesgue-measurable. We denote the Lebesgue measure by $\lambda$. By null set we mean sets of Lebesgue measure zero.

\section{Symbolic dynamics of piecewise linear Markov interval maps} \label{sec1}
In this section we define a class of dynamical systems which includes the original $(\mathbb{R}_{\geq 0}, T)$ system. We apply the well known methods of symbolic dynamics to these systems for use in further proofs. Later we apply the general theory to the $(\mathbb{R}_{\geq 0}, T)$ system.
\subsection{$\ell$ - Markov partitions}
The concept of Markov partitions is studied abundantly in dynamics literature. Places where it is defined include \cite{Bowen1979, Bugiel1985, Bugiel1998} and section 9.1 of \cite{Boyarsky}. We are particularly interested in maps which are piecewise linear, thus the definition we will give resembles class $\mathcal{L}_M$ appearing in section 9.1 of \cite{Boyarsky}.\medskip

Definitions in literature often, however not always assume the compactness of the domain or the finiteness of the partition. Both compactness and finiteness is assumed in Section 9.1 of \cite{Boyarsky} for the class $\mathcal{L}_M$, however in \cite{Bugiel1998} countable Markov partitions are defined on unbounded intervals. We will also use an unbounded, non-compact definition, however we will also assume piecewise linearity.

\begin{defin}[$\ell$-Markov partition] \label{Markovpart}
Given a measurable subset $X$ of the reals, which is a countable union of closed subintervals with non-empty interiors, and a measurable $f:\: X\rightarrow X$ the dynamical system $(X, f)$ has an $\ell$-Markov partition if there exists a countable set $\mathscr{H}$ of closed subintervals of $\mathbb{R}$ with non-empty interior such that the following are true:
\begin{enumerate}
\item For all $A, B \in \mathscr{H}$ if $A \neq B$ then $\interior (A) \cap \interior (B) = \emptyset$.
\item $X = \cup \mathscr{H}$.
\item For all $A \in \mathscr{H}$ there exists $\emptyset \neq \mathscr{A} \subseteq \mathscr{H}$ such that $\interior(f(A)) = \interior (\cup \mathscr{A})$.
\item The function $f$ is linear on the interior of all $H \in \mathscr{H}$.
\end{enumerate}
\end{defin}
 
We introduce a vital definition for symbolic dynamics.

\begin{defin}
Let $(X,f)$ be a dynamical system with an $\ell$-Markov partition $\mathscr{H}$ as in Definition \ref{Markovpart}. Let the elements of $\mathscr{H}$ be $H_i$ where $i$ is a positive integer. Then for a sequence of positive integers of not necessarily finite length $(i_0, i_1, ...)$ define:\[
H_{i_0i_1...} = \bigcap_{n=0}^{\infty} f^{-n}(H_{i_n}).
\]
If the sequence is finite then we only take finite intersections.
\end{defin}

\begin{defin}
Let us define the set $D_0 = \{x \in X: \exists i \in \mathbb{Z}^+: x \in \partial H_i \}$. Let $D$ be the set $\bigcup_{k=0}^\infty f^{-k}(D_0)$. Let $\widetilde{A}$ be $A\setminus D$ for all $A \subseteq X$.
\end{defin}

\begin{cor}
The set $D$ is countable, thus $\lambda(A \Delta \widetilde{A})=0$ for all $A \in \mathscr{L}(X)$. On all $\widetilde{H_i}$ the function $f$ is a homeomorphism between $\widetilde{H_i}$ and $f(\widetilde{H_i})$.
\end{cor}

We state some simple statements to show the structure of the defined sets.

\begin{lemma} \label{Rendszer2}
Given a finite sequence of positive integers $(i_0,...,i_n)$ the following are true for $H_{i_0...i_n}$:
\begin{enumerate}
\item For all $H_i \in \mathscr{H}$ we have that $f(\interior (H_i)) = \interior (f(H_i))$.

\item $\interior (H_{i_0...i_n})$ is an interval.

\item If $\interior (H_i) \cap f^{-1}(H_j) \neq \emptyset$ then $\interior(H_j) \subseteq \interior (f(H_i))$.

\item If $H_{i_0...i_n}$ has a non-empty interior then $f^k(\interior (H_{i_0...i_n})) = \interior (H_{i_k...i_n})$ for all integers $1 \leq k\leq n$.

\item For all $(j_0...j_n) \neq (i_0...i_n)$ sequences of positive integers $\interior (H_{i_0...i_n}) \cap \interior (H_{j_0...j_n}) = \emptyset$.

\item $\interior (H_{i}) = \interior ( \bigcup \{ H_{ik}: \interior (H_k) \subseteq f(H_{i}) \})$.

\item $\interior (H_{i_0...i_n}) = \interior ( \bigcup \{ H_{i_0...i_nk}: \interior (H_k) \subseteq f(H_{i_n}) \})$.

\item For all $k\in \mathbb{Z}^+$ the intersection $\interior (H_{i_0...i_nk}) \cap  \interior (H_{i_0...i_n})$ is non-empty if and only if 
$k \in \{m:\interior(H_m) \subseteq f(H_{i_n}) \}$.

\item If $( \interior (H_{i_0...i_nk}) \cap \interior (H_{i_0...i_n}) ) \neq \emptyset$ then $\lambda(H_{i_0...i_nk}) = \frac{\lambda(H_{k})}{|f'(H_{i_n})|} \lambda(H_{i_0...i_n})$.
\end{enumerate}
\end{lemma}

\begin{proof}
It is worth noting that $H_{i_0...i_n}$ is the set of points for which for all $0\leq k \leq n$ integers $f^k(x) \in H_{i_k}$.
\begin{enumerate}
\item We have that $f$ is a homeomorphism on $\interior (H_i)$ and that $f(H_i) = f(\interior (H_i)) \cup f(\partial H_i)$. Since $f(\partial H_i)$ consists of at most two points we have that $\interior (f(\interior (H_i)) \cup f(\partial H_i)) = \interior (f(\interior (H_i))) = f(\interior (H_i))$. 

\item We have that $\interior (H_{i_0...i_n}) = \interior (H_{i_0}) \cap \interior (f^{-1}(H_{i_1...i_n}))$. On $\interior (H_{i_0})$ the function $f$ is a homeomorphism, so we also have that $\interior (H_{i_0}) \cap \interior (f^{-1}(H_{i_1...i_n})) =  \interior (H_{i_0}) \cap f^{-1}(\interior (H_{i_1...i_n}) )$, and that the preimage of an interval is an interval, thus if $\interior (H_{i_1...i_n})$ is an interval, then $\interior (H_{i_0...i_n})$ is also an interval. Since $\interior (H_i)$ is always an interval, this proves the statement by induction for all cases.

\item We have that $\interior (H_i) \cap f^{-1}(H_j) \neq \emptyset$ if and only if $f(\interior (H_i) \cap f^{-1}(H_j)) \neq \emptyset$. Since $f$ is a homeomorphism on $\interior (H_i)$ this is equivalent to $\interior (f(H_i)) \cap H_j \neq \emptyset$. From Definition \ref{Markovpart} we have that there exists $\emptyset \neq \mathscr{A} \subseteq \mathscr{H}$, such that $\interior (f(H_i)) = \interior (\cup \mathscr{A})$. It follows, that there exists $H_k \in \mathscr{A}$ such that $\interior (H_k) \cap H_j \neq \emptyset$. From Definition \ref{Markovpart} this implies that $H_k = H_j$ and thus $\interior(H_j) \subseteq \interior (f(H_i))$.

\item It suffices to show that if $\interior (H_{i_0...i_n})\neq \emptyset$ then $f(\interior (H_{i_0...i_n})) = \interior (H_{i_1...i_n})$. This is equivalent to $f(\interior (H_{i_0}) \cap \interior (f^{-1}(H_{i_1...i_n}))) = \interior (H_{i_1...i_n})$. Since $f$ is a homeomorphism on $\interior (H_{i_0})$ we have that this is equivalent to $\interior (f(H_{i_0})) \cap \interior (H_{i_1...i_n}) = \interior (H_{i_1...i_n})$. We have that $\interior (H_{i_0...i_n}) \neq \emptyset$, thus $\interior (H_{i_0}) \cap f^{-1}(H_{i_1})\neq \emptyset$. This implies that $\interior (H_{i_1}) \subseteq \interior (f(H_{i_0}))$. This proves the statement.

\item For some $k$ we have that $i_k \neq j_k$. In this case $\interior (H_{i_k}) \cap \interior (H_{j_k}) = \emptyset$. We have that $f^{k}(\interior (H_{i_0...i_n})) \cap f^{k}(\interior (H_{j_0...j_n})) = \interior (H_{i_k...i_n}) \cap \interior (H_{j_k...j_n}) = \emptyset$. This proves the statement.

\item Let $\mathscr{A} = \{ H_{k}: \interior (H_k) \subseteq f(H_{i}) \}$. From Definition \ref{Markovpart} we have that $\interior (f(H_i)) = \interior (\cup \mathscr{A})$. Let us take the preimage of both sides. Now we have that $\interior (H_i) \subseteq f^{-1}(f(\interior (H_i))) = f^{-1}(\interior (f(H_i))) = f^{-1}(\interior (\cup \mathscr{A}))$. In summary we have that $\interior (H_i) \subseteq f^{-1}(\interior (\cup \mathscr{A}))$. The left hand side of this is an open set, so by taking the interior of the right hand side, the inclusion is not violated, so $\interior (H_i) \subseteq \interior (f^{-1}(\interior (\cup \mathscr{A})))$. This implies that $\interior (H_i) \subseteq \interior (f^{-1}( \cup \mathscr{A}))$. Let us denote $\{ H_{ik}: \interior (H_k) \subseteq f(H_{i}) \}$ as $\{H_{ik}\}$. We have that $\interior (\cup \{H_{ik}\}) = \interior (H_{i}) \cap \interior (f^{-1}(\cup \mathscr{A})) = \interior (H_i)$.

\item We have that $\interior (H_{i_0...i_n}) = \interior (H_{i_0...i_{n-1}}) \cap \interior (f^{-n}(H_{i_n}))$. Since $f^{n}$ is a homeomorphism on $\interior (H_{i_0...i_{n-1}})$  thus we can write that $\interior (H_{i_0...i_n}) = \interior (H_{i_0...i_{n-1}}) \cap f^{-n}(\interior (H_{i_n}))$. Since $\interior (H_{i_n}) = \interior (\cup \{H_{i_nk}:\interior (H_k) \subseteq f(H_{i_n})\})$ the statement follows easily.
\end{enumerate}
The rest follows straightforward from the definitions and previous statements of this lemma.
\end{proof}

\subsection{Expansivity}
We need some stronger assumptions for certain statements, than the existence of $\ell$-Markov partitions. We introduce expansive $\ell$-Markov partitions. On a finite partition with more than one element on a bounded domain this is a property implied by $\ess\inf_{x \in X} |f'(x)| > 1$. In the literature one can find numerous definitions of 'expanding' mappings such as in \cite{Bowen1979}, condition 1.H1 in \cite{Bugiel1985} and \cite{Bugiel1991},  condition 4.2.H8 in \cite{Bugiel1998} and among the conditions of Theorem 9.4.2 in \cite{Boyarsky}. This condition is often used for proving the existence of certain kinds of absolutely continuous invariant measures. Depending on the exact theorem stated the specific condition changes from author to author. We also introduce a slightly weaker property for the same purposes. This property will allow us to prove stronger properties of the symbolic representation of the dynamics as well.\medskip

For $A,B \subseteq X$ we use the notation $A \Delta B = (A \setminus B) \cup (B \setminus A)$ for the symmetric difference.
\begin{defin} \label{expansdef}
The dynamical system $(X,f)$ is said to have an expansive $\ell$-Markov partition, if for all $\delta > 0$ and $I \subseteq X$ bounded interval there is a finite set of positive integer sequences of finite length $(i^{(1)}_0,..., i^{(1)}_{n_1}), ... , (i^{(N)}_0,..., i^{(N)}_{n_N})$ for which: \[
\lambda \Big( I \Delta \bigcup_{k=1}^N H_{i^{(k)}_0...i^{(k)}_{n_k}} \Big) < \delta.
\]
\end{defin}

We remark that if the measure of the sets of the partition is bounded and there is some $\varepsilon > 0$ that $|f'(H_i)|>1+\varepsilon$ for all $H_i \in \mathscr{H}$, then $\mathscr{H}$ is expansive.
 
\begin{defin}
Let the symbolic topology $\tau_\sigma$ on $\widetilde{X}$ be the one generated by the empty set, and the elements of $\{\widetilde{H}_{i_0...i_n}: n \in \mathbb{Z}^+, (i_0,...,i_n)\in (\mathbb{Z}^+)^n\}$.
\end{defin}

\begin{theorem} \label{topologia}
If the $\ell$-Markov partition is expansive, then the symbolic topology is the standard subspace topology $\tau$ on $\widetilde{X}$ inherited from $\mathbb{R}$.
\end{theorem}

\begin{proof}
All $\interior (H_{i_0...i_n})$ are open intervals and we have that $\widetilde{H}_{i_0...i_n} = \interior (H_{i_0...i_n}) \cap \widetilde{X}$. This implies that $\widetilde{H}_{i_0...i_n} \in \tau$, thus $\tau_\sigma \subseteq \tau$. We also prove that $\tau \subseteq \tau_\sigma$.

\begin{lemma} \label{szepar}
If the $\ell$-Markov partition is expansive, then for any infinite positive integer sequence $i_0, i_1, ...$ the set $\widetilde{H}_{i_0...}$ is either empty or has exactly one element.
\end{lemma}

\begin{proof}
Proceeding towards a contradiction let us suppose that $x\neq y$ and $x,y \in \widetilde{H}_{i_0...}$. Now $\widetilde{(x, y)} \subseteq \widetilde{H}_{i_0...}$. This implies that for all finite positive integer sequences $(j_0, ..., j_n)$ we have, that either $\widetilde{(x,y)} \subseteq \widetilde{H}_{j_0...j_n}$ or $\widetilde{(x,y)} \cap \widetilde{H}_{j_0...j_n} = \emptyset$. This contradicts that $\mathscr{H}$ is an expansive Markov partition, since the condition described in Definition \ref{expansdef} fails for any proper subinterval of $(x,y)$.
\end{proof} 

Let us consider a $\emptyset \neq G \in \tau$, such that for some $a, b \in \mathbb{R}$ we have that $\widetilde{(a,b)} = G$. Now for any $x \in G$ we have $c,d \in G$ such that $c<x<d$. Since $\widetilde{H}_{x,f(x),f^2(x)...}$ can only have one element, there exists some $k \in \mathbb{N}$ such that $\widetilde{H}_{x...f^k(x)} \cap \{c,d\} = \emptyset$ thus $\widetilde{H}_{x...f^k(x)} \subset G$. This implies that $G \in \tau_\sigma$ thus $\tau \subseteq \tau_\sigma$. 
\end{proof}

\begin{theorem} \label{expans2}
The $\ell$-Markov partition $\mathscr{H}$ of the dynamical system $(X,f)$ is expansive if and only if the following criterion holds. For all $\{ i_n \}_{n=0}^\infty \in (\mathbb{Z}_{> 0} )^\mathbb{N}$ such that $\tilde{H}_{i_0 i_1 ...} \neq \emptyset$ we have the following:
\begin{equation} \label{shrinking}
0 = \lim_{n\to \infty} \lambda (H_{i_0 i_1 ... i_n})
\end{equation}
\end{theorem}

\begin{proof}
We first prove that expansivity implies \eqref{shrinking}. By Lemma \ref{Rendszer2} we have the following two equalities: 
$$H_{i_0 i_1 ... i_n} = \bigcap_{j=0}^n H_{i_0 ... i_j}, \;\; H_{i_0 i_1 ... } = \bigcap_{j=0}^\infty H_{i_0 ... i_j}.$$
This is a shrinking sequence of sets, thus it is also true that:
$$\interior(H_{i_0 i_1 ... i_n}) = \bigcap_{j=0}^n \interior(H_{i_0 ... i_j}).$$
If for no $n$ is $\interior(H_{i_0 i_1 ... i_n})$ bounded, then  $\overline{ \bigcap_{j=0}^\infty \tilde{H}_{i_0 ... i_j} } = \overline{ \tilde{H}_{i_0i_1...}}$ is either an unbounded interval, or the empty set. From Lemma \ref{szepar} if expansivity applies to $\mathscr{H}$, then it follows that $\overline{\tilde{H}_{i_0i_1...}}$ must be empty. However we have assumed in the statement of the theorem, that $\tilde{H}_{i_0 i_1 ...} \neq \emptyset$. Thus we may assume that from a large enough $n$ we have that $\lambda (H_{i_0 i_1 ... i_n}) < \infty$, thus from the continuity of the Lebesgue measure we have that: $$\lim_{n\to \infty} \lambda (H_{i_0 i_1 ... i_n}) = \lambda (H_{i_0i_1 ...}) = \lambda(\interior (H_{i_0i_1 ...})).$$
By Lemma \ref{szepar} we have that expansivity implies \eqref{shrinking}.\medskip

For the other implication we are given any $\delta$ and a bounded interval $I$. Since the set $\tilde{X}$ is dense in $X$ we can find $a, b \in \tilde{X}$ such that $\lambda(I \Delta (a,b)) < \nicefrac{\delta}{3}$. We state a lemma. 

\begin{lemma} \label{iti}
For all $x \in \tilde{X}$ there exists a positive integer sequence $i_0 i_1 ...$ such that $x \in \tilde{H}_{i_0 i_1 ...}$
\end{lemma}
\begin{proof}
We have that $x \in H_{i_0}$ for some $i_0$, and if $x \in H_{i_0 ... i_n}$ we can pick $i_{n+1}$ such that $x \in H_{i_0 ... i_{n+1}}$. Thus we can pick a positive integer sequence $i_0 i_1 ...$ such that $x \in H_{i_0 i_1 ...}$. Since $x \in \tilde{X}$ we also have that $x \in \tilde{H}_{i_0 i_1 ...}$
\end{proof}
Let us choose $n_\delta$ such that there exists $(i_0, i_1, ... i_{n_\delta}), (j_0, j_1, ... j_{n_\delta}) \in \mathbb{Z}_{> 0}^{n_\delta}$ such that $a \in \tilde{H}_{i_0 ... i_{n_\delta}}, b \in \tilde{H}_{j_0 ... j_{n_\delta}}$ and $\lambda(\tilde{H}_{i_0 ... i_{n_\delta}}) + \lambda(\tilde{H}_{j_0 ... j_{n_\delta}}) < \nicefrac{\delta}{3}$. By Lemma \ref{iti} and by \eqref{shrinking} there always exist such sequences. For our cover let us pick sets indexed by $n_\delta$ long integer sequences. Let  $\mathscr{A}$ be the set of all integer sequences $(a_0, ..., a_{n_{\delta}})$, for which $H_{a_0...a_{n_{\delta}}} \neq \emptyset$ and $\interior(H_{a_0...a_{n_\delta}})$ lies between $\interior (H_{i_0 ... i_{n_\delta}})$ and $\interior (H_{j_0 ... j_{n_\delta}})$. This is a countable set. Let us index it such that $\mathscr{A} = \{(a^{(i)}_0, ..., a^{(i)}_{n_\delta})\}_{i \in \mathbb{N}}$. We have the following:
$$\lambda \Big( I \Delta \bigcup_{i=0}^\infty H_{a^{(i)}_0 .. a^{(i)}_{n_\delta}} \Big)\leq \frac{2}{3} \delta.$$
Thus for a large enough $N$ we have:
$$\lambda \Big( I \Delta \bigcup_{i=0}^N H_{a^{(i)}_0 .. a^{(i)}_{n_\delta}} \Big)\leq  \delta.$$
\end{proof}

For all the remaining statements in this subsection we assume that we have an expansive $\ell$-Markov partition.

\begin{defin}
Let the function $\underline{i}: \widetilde{X} \rightarrow \mathbb{Z}^\mathbb{N}$ be the itinerary of a point. For a point $x\in \widetilde{X}$ it gives the integer sequence $(i_0, ...)$ for which $x \in \widetilde{H}_{i_0...}$. 
\end{defin}

Since we are in $\widetilde{X}$ it is straightforward to see from Lemma \ref{Rendszer2} and Theorem \ref{topologia} that $\underline{i}$ is a well-defined bijection between $\widetilde{X}$ and $\underline{i}(\widetilde{X})$.

\begin{defin} \label{trajspace}
We denote by $\Omega_f = \{\underline{i}(x):\: x \in \widetilde{X}\}$ the space of allowed symbolic trajectories. We define the metric $d: \Omega_f \times \Omega_f \rightarrow \widetilde{X}$ such, that $d((i_0,...), (j_0,...)) = 2^{-n}$ where $n$ is the smallest non-negative integer for which $i_n \neq j_n$.
\end{defin}

\begin{theorem} \label{Homeo}
The function $\underline{i}$ is a homeomorphism between $\widetilde{X}$ and $\underline{i}(\widetilde{X}) = \Omega_f$.
\end{theorem}

\begin{proof}
Let $\underline{i}(x) = (i_0,...)$ for some $x \in \widetilde{X}$. Let us start by proving the continuity of $\underline{i}$. Then $x \in \interior (H_{i_0...i_n})$ for all $n \in \mathbb{N}$. From here given an $\varepsilon > 0$ and an $N \in \mathbb{N}$ for which $2^{-N} < \varepsilon$ let us pick a $\delta>0$ for which $U = B(x, \delta) \subseteq \interior (H_{i_0...i_N})$. This ensures that $\underline{i}(\widetilde{U}) \subseteq B(\underline{i}(x), \varepsilon)$.\medskip

For the continuity of the inverse we have that $\underline{i}(\widetilde{H}_{i_0...i_n}) = B(\underline{i}(x), 2^{-1-n})$. From Theorem \ref{topologia} we have that for all $\varepsilon > 0$ we have $N \in \mathbb{N}$ such that $\widetilde{H}_{i_0...i_N} \subset B(x, \varepsilon)$. It follows that for $0 < \delta < 2^{-1-N}$ we have that $\underline{i}^{-1}(B(\underline{i}(x), \delta)) \subseteq B(x, \varepsilon)$. 
\end{proof}

\begin{cor} \label{conju}
The dynamical systems $(\widetilde{X}, f)$ and $(\Omega_f, \sigma)$ are topologically conjugate.
\end{cor}

\begin{theorem}
If for some $x \in \widetilde{X}$ the orbit of $\underline{i}(x)$ is periodic for the shift operator then the orbit of $x$ is periodic as well.
\end{theorem}

\begin{proof}
If $\underline{i}(x)$ is $i_0...i_ni_0...i_n...$, then $x \in \widetilde{H}_{i_0...i_ni_0...i_n...}$ and $f^{n+1}(x) \in \widetilde{H}_{i_0...i_ni_0...i_n...}$. This means that $x = f^{n+1}(x)$, so the orbit of $x$ has period $n+1$.
\end{proof}

\subsection{Application to $(\mathbb{R}_{\geq 0}, T)$}

We apply the previously introduced and formalized methods to $(\mathbb{R}_{\geq 0}, T)$. We also show some additional structure, which is only observable in this specific case.

\begin{defin}
We define the closed integer intervals to be $I_n = [n-1, n]$ for any $n\in \mathbb{Z}^+$.
\end{defin}

\begin{defin} \label{symbolinterval}
Given a finite sequence of positive integers $(i_0,...,i_N)$ we denote $I_{i_0...i_N} = \bigcap_{n=0}^{N} T^{-n}(I_{i_n})$. Given an infinite sequence of positive integers $(i_0,...)$ we denote 
$I_{i_0...} = \bigcap_{n=0}^{\infty} T^{-n}(I_{i_n})$.
\end{defin}

\begin{lemma}\label{Rendszer}
Given a finite sequence of positive integers $(i_0,...i_n)$ the following are true for $I_{i_0...i_n}$:
\begin{enumerate}
\item The sets $I_{i_0...i_nk}$ where $k \in \mathbb{Z} \cap [1,4^{\ell(i_n)}+1]$ are placed in a row in increasing or decreasing order. (See Figure \ref{rendabra})

\item The above sets are in increasing order if and only if there is an even number of even numbers among $i_0 ... i_n$. (See Figure \ref{rendabra})
\end{enumerate}
\end{lemma}

\begin{figure}[h]
\includegraphics[width= 10cm]{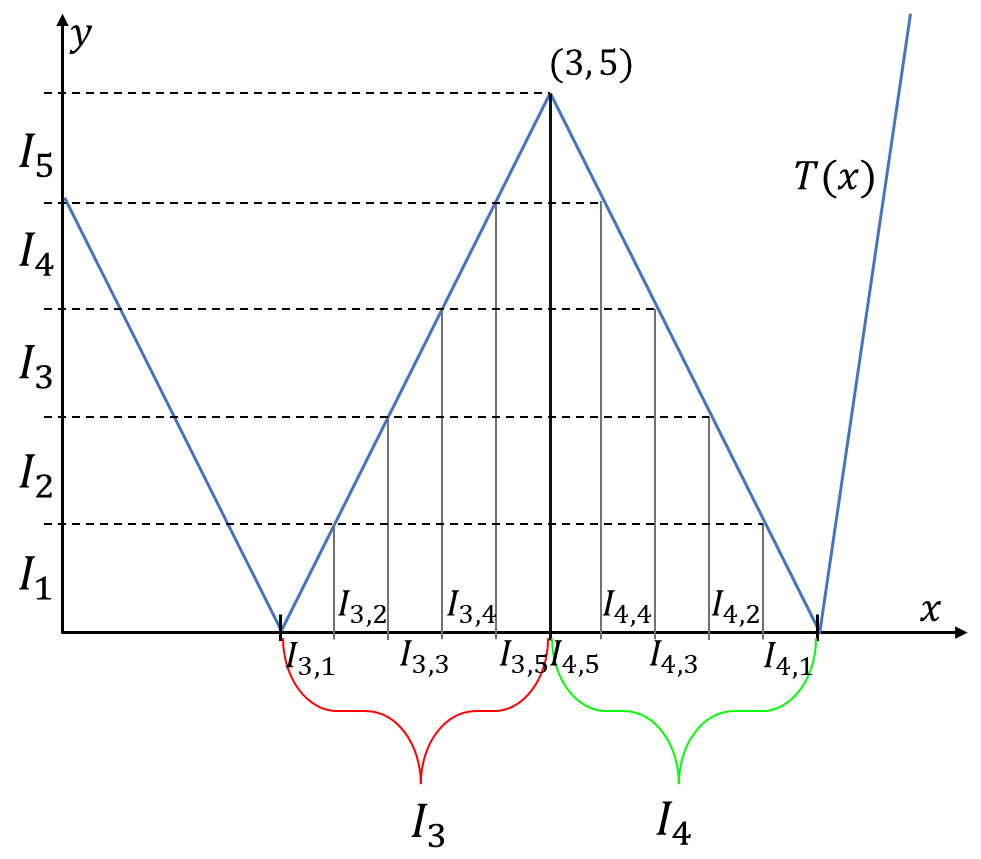}
\centering
\caption{Some labelled intervals of $\mathbb{R}_{\geq 0}$.}
\label{rendabra}
\end{figure}

\begin{theorem}
The set of integer intervals forms an expansive $\ell$-Markov partition of $(\mathbb{R}_{\geq 0}, T)$.
\end{theorem}

\begin{proof}[Proof of $\ell$-Markov partition]
We of course have that $\{I_n\}$ is a partition of $\mathbb{R}_{\geq 0}$. We also have linearity of $T$ on these intervals. The image $T(I_n)$ is $[0,4^{\ell(n)}+1] = \bigcup_{i=1}^{\ell(n)+1} I_n$. Thus all the conditions hold.
\end{proof}

\begin{proof}[Proof of expansivity]
All elements of the partition have unit measure. We also have that $|T'(x)| \geq 5$. Thus by Lemma \ref{Rendszer2} we have that for any $i_0, ..., i_n$ the measure $\lambda(I_{i_0...i_n})$ is at most $5^{-n}$. By Theorem \ref{expans2} we have that the partition is expansive.
\end{proof}

Now we know from Theorem \ref{Homeo} that there exists a symbolic space $\Omega_T$ and a topological conjugacy between $(\tilde{\mathbb{R}}_{\geq 0},T)$ and $(\Omega_T, \sigma)$. We describe $\Omega_T$.

\begin{theorem}
The space $\Omega_T$ is composed of exactly those sequences $(i_0, i_1, ...) \in (\mathbb{Z}_{>0})^\mathbb{N}$, which satisfy the following conditions:
\begin{enumerate}
\item For all $n > 0$ we have that $i_{n+1} \in \{1, 4^{\ell(i_n)}+1\}$
\item There exists no $n \in \mathbb{N}$ such that $\sigma^n((i_0, i_1, ...)) = (1,1,1,...)$ or $\sigma^n((i_0, i_1, ...)) = (4^{\ell(i_{n-1})}+1, 4^{\ell(i_{n-1})+1}+1, 4^{\ell(i_{n-1})+2}+1, ...)$ 
\end{enumerate}
\end{theorem}

\begin{proof}
We have that $f(I_n) = [0,4^{\ell(i_n)}+1]$, thus if $I_{i_0 ... i_n}$ is non-empty, then $I_{i_0 ... i_{n+1}}$ is non-empty exactly when $i_{n+1} \in \{1, ..., 4^{\ell(i_n)}+1\}$. Thus, if the first condition applies, then $I_{i_0 i_1 ...}$ is non-empty from Cantor's axiom. The only way $\tilde{I}_{i_0 i_1 ...}$ can be empty, if the only element of $I_{i_0 i_1 ...}$ eventually steps on an integer. We have that $\{0\} = I_{1,1,1 ...}$ and that $\{4^{\ell(i_{n-1})}+1\} = I_{4^{\ell(i_{n-1})}+1, 4^{\ell(i_{n-1})+1}+1, 4^{\ell(i_{n-1})+2}+1, ...}$. Thus all sequences removed by the second condition are itineraries of points eventually hitting an integer. Furthermore all integers are mapped to either $0$ or some number of the form $4^k + 1$, thus the itineraries removed by the second condition are exactly the itineraries of points eventually hitting integers.
\end{proof}

\section{Shadowing} \label{shadow}

Shadowing of piecewise linear maps of a compact interval onto itself has been extensively studied by multiple authors. From \cite{Tent} we know that even for tent maps the shadowing property does not hold for an uncountable number of examples. Our system is in many ways similar to tent maps discussed in \cite{Tent}, however we have a non-compact domain, which is generally not considered in the literature. We will show, that our system has the shadowing property.\medskip 

For this section we will have to define the itinerary of not just a point, but a sequence of non-negative reals in general. To make it distinct from the function $\underline{i}$ defined previously we denote this by $\underline{j}$.

\begin{defin}
Let $\underline{j}:\: \mathbb{R}_{\geq 0}^\mathbb{N} \rightarrow \mathbb{Z}_{>0}^\mathbb{N}$ be the symbolic trajectory or itinerary of a sequence: for $x = (x_0, x_1, ...) \in \mathbb{R}_{\geq 0}^\mathbb{N}$ the sequence $\underline{j}(x) = j = (j_0, j_1, ...) \in \mathbb{Z}_{>0}^\mathbb{N}$ denotes the lexicographically smallest sequence of positive integers such that $x_k \in I_{j_k}$ for all $k \in \mathbb{N}$.
\end{defin}

\begin{defin}
For some $H \subset \mathbb{R}$ and $\delta \geq 0$ let us define: $$B(H, \delta)= \{x \in \mathbb{R}: \exists y \in H,\, |x-y| < \delta \}. $$
\end{defin}

\begin{lemma} \label{exp}
For all bounded $H \subset \mathbb{R}_{\geq 0}$ and any $\varepsilon \geq 0$ it is true that $$f(B(H, \varepsilon)) \subseteq B \Big(f(H), \varepsilon \cdot \sup_{y \in B(H,\varepsilon)} |f'(y)| \Big)$$ for any continuous $f: \mathbb{R}_{\geq 0} \rightarrow \mathbb{R}_{\geq 0}$, which is differentiable at all but finitely many points of any bounded subinterval of $\mathbb{R}_{\geq 0}$.
\end{lemma}

\begin{proof}
If $k = \sup_{y \in B(H,\varepsilon)} |f'(y)|$ then $f$ is $k$-Lipschitz on $B(H,\varepsilon)$. If $x \in B(H,\varepsilon)$, then there is some $y \in H$ such that $|y-x| < \varepsilon$. By the Lipschitz property $|f(x) - f(y)| \leq k |x-y| < k \varepsilon$. This means that $f(x) \in B(f(H), k\varepsilon)$.
\end{proof}

\begin{lemma} \label{contr}
For any $I_n$ for any $H \subset \mathbb{R}_{\geq 0}$ such that $H \subseteq T(I_n)$ the following holds. For any $\varepsilon > 0$ we have that $T^{-1}(B(H,\varepsilon)) \cap I_n \subseteq B(T^{-1}(H)\cap I_n, \frac{\varepsilon}{|T'(I_n)|})\cap I_n$.
\end{lemma}

\begin{proof}
Let us pick any $x \in T^{-1}(B(H,\varepsilon)) \cap I_n$. Then $x\in I_n$ and $T(x) \in B(H,\varepsilon)$. Let us pick a $z \in H$ such that $|z-T(x)|< \varepsilon$. Since $H \subseteq T(I_n)$ there exists $y \in T^{-1}(H) \cap I_n$ such that $T(y) = z$. Let $J$ be the closed interval with endpoints $T(x)$ and $T(y)$ in the correct order. We have already established that the length of $J$ is less than $\varepsilon$. As $T(I_n)$ is an interval, and both $T(x)$ and $T(y)$ are in $T(I_n)$ we have that $J \subseteq T(I_n)$. As $T$ is linear on $I_n$ with slope $T'(I_n)$ we have that $T^{-1}(J) \cap I_n$ is an interval with length less then $\frac{\varepsilon}{|T'(I_n)|}$ and endpoints $x$ and $y$. Thus $x \in B(T^{-1}(H)\cap I_n, \frac{\varepsilon}{|T'(I_n)|})\cap I_n$ as required.
\end{proof}

\begin{defin}
For a dynamical system $(X, f)$ the sequence $\{x_k\}_{k=0}^\infty$ of elements of $X$ is called a $\delta$-pseudo orbit of $f$, if $x_{k+1} \in B(f(x_k), \delta)$ for all $k \in \mathbb{N}$.
\end{defin}

\begin{defin}
Given two sequences of elements of a metric space $x_k$ and $y_k$ it is said that the sequence $y_k$ $\varepsilon$-shadows the sequence $x_k$, if $y_k \in B(x_k, \varepsilon)$ for all indices $k$. 
\end{defin}

\begin{defin}
We say that a dynamical system $(X, f)$ has the shadowing property, if for every $\varepsilon > 0$ there exists a $\delta > 0$ such that every $\delta$-pseudo orbit is $\varepsilon$-shadowed by at least one true orbit.
\end{defin}

\begin{defin}
Let us define the space $\Omega'_T$ to be composed of exactly those sequences of positive integers $(i_0,i_1,...)$ which satisfy that $i_{n+1} \in [1, 4^{\ell(i_{n})}+1]$ for any non-negative integer $n$. 
\end{defin}

For the sequences in $\Omega_T$ this applies automatically, thus $\Omega_T \subseteq \Omega'_T$.

\begin{cor} \label{megengedett}
For all $(i_0,i_1,...) \in \Omega_T'$ for all $n \in \mathbb{N}$ we have that $I_{i_{n+1}} \subseteq T(I_{i_n})$.
\end{cor}

\begin{theorem}
Let $\{x_k\}_{k=0}^\infty$ be a $\delta$-pseudo orbit. Let $\underline{j}(x)=(i_0, i_1, ...)$. If $\underline{j}(x) \in \Omega'_T$, then for $y \in I_{i_0i_1...}$ the sequence $(y, T(y), T^2(y),...)$, $\varepsilon$-shadows $x_k$ for $\varepsilon = \frac{\delta}{4}$.
\end{theorem}

\begin{proof}

By the definition of a pseudo orbit $x_{n+1} \in B(T(x_n), \delta)$ for any non-negative integer $n$, so $x_n \in T^{-1}(B(x_{n+1}, \delta))$. From here $x_n \in T^{-1}(B(I_{i_{n+1}}, \delta))$. Also $x_n \in I_{i_n}$. Since $|T'(x)|\geq 5$ for all $x \in \mathbb{R}_{\geq 0}$, then from Corollary \ref{megengedett} and Lemma \ref{contr} $x_n \in I_{i_n} \cap T^{-1}(B(I_{i_{n+1}}, \delta)) \subseteq I_{i_n} \cap B(I_{i_n} \cap T^{-1}(I_{i_{n+1}}), \frac{\delta}{5})$. This is a subset of $B(I_{i_n i_{n+1}},  \frac{\delta}{5})$.\medskip

From here we will prove by induction the following: $x_k\in B(I_{i_ki_{k+1}...i_{k+n}},\sum_{i=1}^n \frac{\delta}{5^i})$ for all $n, k \in \mathbb{N}$.\medskip

We will use induction on $n$. For $n=1$ we have just proven the statement for all $k$. We assume that for some $n < N$ the statement holds for every $k$, and we prove that it also holds for $n = N$. Hence $x_{k+1} \in B(I_{i_{k+1}i_{k+2}...i_{k+N}},\sum_{i=1}^{N-1} \frac{\delta}{5^i})$. This is an interval. Now $x_k \in I_{i_k} \cap T^{-1}(B(x_{k+1},\delta)) \subseteq I_{i_k} \cap T^{-1}(B(I_{i_{k+1}i_{k+2}...i_{k+N}},\sum_{i=0}^{N-1} \frac{\delta}{5^i}))$. From Corollary \ref{megengedett} and Lemma \ref{contr} it follows that $I_{i_k} \cap T^{-1}(B(I_{i_{k+1}i_{k+2}...i_{k+N}},\sum_{i=0}^{N-1} \frac{\delta}{5^i})) \subseteq I_{i_k} \cap B(I_{i_k} \cap T^{-1}(I_{i_{k+1}i_{k+2}...i_{k+N}}), \sum_{i=1}^{N} \frac{\delta}{5^i})$. From this it follows that $x_k \in B(I_{i_ki_{k+1}...i_{k+N}},\sum_{i=1}^N \frac{\delta}{5^i})$.\medskip

Since it is a geometric series $\sum_{i=1}^\infty \frac{\delta}{5^i} = \frac{\delta}{4}$. For any non-negative integer $k$, $\{T^k(y)\} = I_{i_ki_{k+1}...}$. From here $x_k \in B(T^k(y), \frac{\delta}{4})$.

\end{proof}

\begin{theorem}
If $\frac{1}{4} > \delta > 0$, then for all $\delta$-pseudo orbits $\{x_n\}_{n=0}^{\infty}$ we have that $\underline{j}(\{T(x_n)\}_{n=0}^{\infty}) \in \Omega'_T$.
\end{theorem}

\begin{proof}
Let $\underline{j}(\{T(x_n)\}_{n=0}^{\infty}) = (j_0, j_1, ...)$ and $\underline{j}(\{x_n\}_{n=0}^{\infty})=(i_0, i_1, ...)$. We want to prove that $I_{j_{n+1}} \subset T(I_{j_n})$ for all natural $n$. We have seen that $T(x_n) \in B(I_{i_{n+1}}, \delta)$, so if $\delta < 1$, then $i_{n+1} \in \{j_n - 1,\: j_n,\: j_n + 1\}$. If $i_{n+1} \leq j_n$, then $T(I_{i_{n+1}}) \subseteq T(I_{j_n})$, so $I_{j_{n+1}} \subset T(I_{j_n})$.\medskip

If $i_{n+1} = j_n + 1$, the only time when $T(I_{i_{n+1}}) \subseteq T(I_{j_n})$ does not hold, is when $\ell(j_n + 1) \neq \ell(j_n)$, so $j_n = 4^m$ for some positive integer $m$. In this case if $x_{n+1} \in [4^m, 4^m + \frac{1}{4}]$, then $T(x_{n+1}) \in T(I_{j_n})$, so $I_{j_{n+1}} \subset T(I_{j_n})$. From here if $\delta < \frac{1}{4}$ then $(j_0, j_1...) \in \Omega'_T$.
\end{proof}

\begin{theorem} \label{altorbit}
Let $\{x_n\}_{n=0}^{\infty}$ be a $\delta$-pseudo orbit, with $\delta < \frac{1}{4}$. Let $\underline{j}(\{T(x_n)\}_{n=0}^{\infty}) = (j_0, j_1, ...)$ and $\underline{j}(\{x_n\}_{n=0}^{\infty})=(i_0, i_1, ...)$. Now, if $y\in I_{j_0j_1...}$, then $\{y,\;T(y),\;T^2(y)...\}$, $20\delta$-shadows the sequence $\{T(x_{n})\}_{n=0}^{\infty}$.
\end{theorem}

\begin{proof}
Let us first define some parameters.
\begin{enumerate}
\item Let $h_k$ be the absolute value of the slope of $T$ on $\interior (I_k)$.
\item Let $\eta_k = \sup_{y \in B(x_{k+1},\delta)} |T'(y)| $.
\item Let $a_n^{(k)} = \frac{\eta_{n+k}}{\prod_{l=k}^{k+n} h_{j_l}}$.
\end{enumerate}

We know that $T(x_k) \in B(x_{k+1}, \delta)$. From this $T(T(x_k)) \in T(B(x_{k+1}, \delta))$, so by Lemma \ref{exp} $T(T(x_k)) \in B(T(x_{k+1}), \eta_k \delta)$. Now as a consequence of Corollary \ref{megengedett} and Lemma \ref{contr} $T(x_k) \in I_{j_k} \cap T^{-1}(B(T(x_{k+1}),\eta_k \delta)) \subseteq I_{j_k} \cap T^{-1}(B(I_{j_{k+1}},\eta_k \delta)) \subseteq I_{j_k} \cap B(I_{j_k} \cap T ^{-1}(I_{j_{k+1}}), \frac{\eta_k \delta}{h_{j_k}})$. The last set is equal to $I_{j_k} \cap B(I_{j_kj_{k+1}}, \delta a_0^{(k)})$.\medskip

Let us assume that for all $k$ positive integers and for all $n\leq N$ for some positive integer $N$: 
\begin{equation}\label{kotelezo}
T(x_k) \in B(I_{j_kj_{k+1}...j_{k+n}}, \delta \sum_{l=0}^{n} a_l^{(k)}).
\end{equation} 
We prove that the inclusion \eqref{kotelezo} also holds for $n=N+1$.\medskip

From the assumption $T(x_{k+1}) \in B(I_{j_{k+1}j_{k+2}...j_{k+N+1}}, \delta \sum_{l=0}^{N} a_l^{(k+1)})$. We also know that: \[
T(x_k)\in I_{j_k} \cap B \Bigg( I_{j_k} \cap T^{-1} \big(T(x_{k+1}) \big), \frac{\eta_k \delta}{h_{j_k}} \Bigg).
\] 
Therefore:
\[ T(x_k)\in I_{j_k} \cap B \Bigg( I_{j_k} \cap T^{-1} \Big( B \big(I_{j_{k+1}j_{k+2}...j_{k+N+1}}, \delta \sum_{l=0}^{N} a_l^{(k+1)} \big) \Big), \frac{\eta_k \delta}{h_{j_k}} \Bigg).\] 
By Lemma \ref{contr} we have the following: 
\[
T(x_k)\in I_{j_k} \cap B \Bigg( I_{j_k} \cap B \Big(I_{j_k} \cap T^{-1}(I_{j_{k+1}j_{k+2}...j_{k+N+1}}), \frac{\delta}{h_{j_k}} \sum_{l=0}^{N} a_l^{(k+1)} \Big), \frac{\eta_k \delta}{h_{j_k}} \Bigg).
\]
Therefore:
\[
T(x_k)\in B \Bigg(I_{j_kj_{k+1}...j_{k+N+1}}, \delta \sum_{l=0}^{N+1} a_l^{(k)} \Bigg).\]

This by induction proves that $T(x_k) \in B(I_{j_kj_{k+1}...}, \delta \sum_{l=0}^{\infty} a_l^{(k)}) = B(T^k(y), \delta \sum_{l=0}^{\infty} a_l^{(k)})$ for all integers $k$. If we find an upper bound of $\sum_{l=0}^{\infty} a_l^{(k)}$, then the proof of the theorem is finished.\medskip

From its definition $a_n^{(k)} = \frac{\eta_{n+k}}{\eta_{n+k-1} h_{j_{n+k}}}a_{n-1}^{(k)}$. Since $\ell(x_{n+k+1}) \leq \ell(x_{n+k}) + 1$, $x_{n+k+1} \leq 4^{\ell(x_{n+k})}+1+\delta$, and we have $\delta < 1$, then $\frac{\eta_{n+k}}{\eta_{n+k-1}} \leq 4$. We always have $h_{j_{n+k}} \geq 5$. Now we only need an upper bound on $a_0^{(k)} = \frac{\eta_k}{h_{j_k}}$. We know that $h_{j_k}$ is the absolute value of the slope of $T$ on the integer interval containing $T(x_k)$. Since $T(x_k)\in B(x_{k+1}, \delta)$ this means that $a_0^{(k)}$ is at most four. From here $\sum_{l=0}^\infty a_l^{(k)} \leq \sum_{l=0}^\infty 4(\frac{4}{5})^l = 20$.
\end{proof}

\begin{theorem}
The dynamical system $(\mathbb{R}_{\geq 0}, T)$ has the shadowing property.
\end{theorem}

\begin{proof}
We have seen in Theorem \ref{altorbit}, that for $\delta < \frac{1}{4}$ for every $\delta$-pseudo orbit $\{x_k\}_{k=0}^{\infty}$ we have a $y_x$ for which the sequence $T^k(y_x)$ $20\delta$-shadows the sequence $T(x_{k})$. This means that $x_{k+1} \in B(T^k(y_x), 21\delta)$. Now selecting $z_0 \in T^{-1}(x_0)$ and $z_n = x_{n-1}$ for every positive integer $n$, and we find $y_z$, then the sequence $\{x_k\}_{k=0}^{\infty}$ is $21\delta$-shadowed by the sequence $T^k(y_z)$, a true orbit.
\end{proof}

\section{Orbit of rationals} \label{orbit of rationals}

The representation of real numbers by discrete time dynamical systems is a widely considered concept. We would like to encode numbers by their symbolic trajectories. This serves as a generalization of the $q$-adic expansion, as it is stated for example in \cite{adler}. In our case, with our Markov maps we obtain real numbers as Cantor series. The influential work of Rényi \cite{renyi} on $f$-expansions should be mentioned. It generalizes the $q$-adic expansion of real numbers in a different way.

\subsection{Theoretical results} \label{theoretical_results}

We deal with dynamical systems $(X, f)$, where $X \subseteq \mathbb{R}$. About $f$ further assumptions will be made. In this section we allow integer intervals to have negative indices. To avoid ambiguity we introduce a new notation.

\begin{defin}
Let $J_n := [n-1, n]$ for arbitrary $n \in \mathbb{Z}$. Let $J_{i_0i_1...} := \bigcap_{k=0}^{\infty}f^{-k}(J_{i_k})$ for an arbitrary $i_0, i_1...$ finite or infinite integer sequence. If the sequence is finite, then we only take finite intersection.
\end{defin}

\begin{defin}[Assumptions on $f$] \label{z1}
For $f$ the following assumptions hold:
\begin{enumerate}
\item Some subset $\mathscr{J}$ of the integer intervals forms an $\ell$-Markov partition of $(X,f)$ as in Definition \ref{Markovpart}.
\item For every $J_i \in \mathscr{J}$ the set $\mathscr{A} \subseteq \mathscr{J}$ from Definition \ref{Markovpart}, for which $\interior (f(J_i)) = \interior (\cup \mathscr{A})$ has at least two elements.
\end{enumerate}
\end{defin}

Since all integer intervals have unit Lebesgue-measure and the second assumption ensures that on every integer interval the magnitude of the slope is at least two we have that $\mathscr{J}$ is an expansive $\ell$-Markov partition. In this case a bit weaker form of Lemma \ref{Rendszer} is still true. We also refer to Lemma \ref{Rendszer2}.

\begin{lemma} \label{Rendszer3}
For a given finite sequence of integers $i_0, ..., i_n$ the following are true for $J_{i_0...i_n}$:
\begin{enumerate}
\item The set $\{k: \interior (J_k) \subseteq f(J_{n})\}$ for any integer $n$ is a set of consecutive integers.
\item Consider the sets $J_{i_0...i_nk}$ where $k \in \{m: \interior (J_m) \subseteq f(J_{i_n})\}$ in the order they appear on the real line. We have that their $(n+1)$st index, $k$ either increases or decreases from left to right. It increases if and only if there is an even number of indices $j$ among $i_0, ..., i_n$ such that $f$ is decreasing on $\widetilde{J_j}$.
\item For $a, a+1 \in \{m: \interior (J_m) \subseteq f(J_{i_n})\}$ we have that $\overline{\interior (J_{i_0...i_na})} \cup \overline{\interior (J_{i_0...i_n(a+1)})}$ is an interval.
\item For a point $x \in X$ there 
exists a natural number $k$ for which $f^k(x)\in \mathbb{Z}$ if and only if for that $k$ and for some $n\in \mathbb{Z}$ we have that $f^k(x) \in \partial J_n$.
\item If $\interior (J_{i_0...i_nk}) \cap \interior (J_{i_0...i_n})  \neq \emptyset$ then $\lambda(J_{i_0...i_nk}) = \frac{1}{\#\{m: \interior (J_m) \subseteq f(J_{i_n})\}} \lambda(J_{i_0...i_n})$.
\end{enumerate}
\end{lemma}

\begin{proof}
These statements all follow simply from our assumptions about $f$. 
\end{proof}

\begin{defin}
Let us consider a finite integer sequence $i_0i_1...i_n$ for which $\interior (J_{i_0...i_n})$ is non-empty. If the number of elements $j$ in $i_0...i_n$ for which the function $f$ is decreasing on $J_j$ is even then we say that the direction of $J_{i_0...i_n}$ is forward, otherwise it is backward.
\end{defin}

We would like to understand how the rationals behave in our system. Since $f$ may be discontinuous at any $n \in \mathbb{Z}$, the orbit of integers cannot be meaningfully described. Thus we only consider rationals which have no integers in their orbit. Thus from now on we consider the $(\widetilde{X},f)$ system.

\begin{defin} \label{regret}
Let $L_k(x)$ and $U_k(x)$ be the $k$th lower and upper regressors of $x$ meaning that: \[
L_k(x) = \inf\{y\in \widetilde{X}:\: y \in \widetilde{J}_{\underline{i}(x)|k}\}, \quad
U_k(x) = \sup\{y\in \widetilde{X}:\: y \in \widetilde{J}_{\underline{i}(x)|k}\}.
\]
\end{defin}

\begin{defin} \label{dif_ferenc}
Let $d_k(x) = L_k(x) - L_{k-1}(x)$ for $k\geq1$.
\end{defin}

Suppose that $\widetilde{J}_{i_0...i_n}$ is a nonempty set. We define tags of the nonempty sets of the form $\widetilde{J}_{i_0...i_nk}$.

\begin{defin}
Let $t: \Omega \rightarrow \mathbb{Z}$ be the tag function, where $\Omega$ is the space of finite integer sequences $(j_0, ..., j_n)$, for which $\widetilde{J}_{{j_0...j_n}} \neq \emptyset$. Let $t(i)=i$ for all integers $i$ such that $J_i \in \mathscr{J}$. We define the value of $t(i_0, ... ,i_n, k)$, where $(i_0, ... ,i_n, k) \in \Omega$. We use the following variables: 
$$N = \#\{p: \interior (J_p) \subseteq f(J_{i_n})\},$$  $$M = \max \{p: \interior (J_p) \subseteq f(J_{i_n})\},\: m = \min \{p: \interior (J_p) \subseteq f(J_{i_n})\}.$$ 
Let the function $t$ be a bijection between elements of $\Omega$ of the form $(i_0, ... ,i_n, k)$, and the integers $\{0,...,N-1\}$. If $\widetilde{J}_{i_0...i_n}$ is forward then for $a \in \{0,...,N-1\}$ we have $a = t(i_0, ..., i_n, (a+m))$ and if $\widetilde{J}_{i_0...i_n}$ is backward then $a= t(i_0, ..., i_n, (M-a))$. It is clear that $t$ only depends on the sequence given as an input.
\end{defin}

\begin{theorem} \label{regrDif}
Let $\underline{i}(x) = (i_0,i_1...)$ for some $x \in \widetilde{X}$. For some $a_n, M_n \in \mathbb{Z}$ we have that $d_n(x) = \frac{a_n}{M_n}$ where $M_n = \prod_{k=0}^{n-1} \#\{p: \interior (J_p) \subseteq f(J_{i_k})\} = \prod_{k=0}^{n-1} |f'(J_{i_k})|$ and $a_n = t(i_0, ..., i_n)$ is the tag of $\widetilde{J}_{i_0...i_n}$ as a subset of $\widetilde{J}_{i_0...i_{n-1}}$.
\end{theorem}

\begin{proof}
We have by Lemma \ref{Rendszer3} that $\lambda (\widetilde{J}_{i_0...i_{n-1}})$ is always $\nicefrac{1}{M_{n-1}}$ if it is nonempty. If $\overline{\widetilde{J}_{i_0...i_{n-1}}} = [x,y]$ for $x,y \in \mathbb{R}$ then from the definition of the tag we have that $\overline{\widetilde{J}_{i_0...i_n}} = [x + \frac{a_n}{M_n}, x + \frac{a_n+1}{M_n}]$, which implies the statement.
\end{proof}

\begin{cor} \label{regreq}
We have that $\lim_{k\rightarrow \infty} L_k(x) = x$ and $L_0(x) + \sum_{k=1}^\infty d_k(x) = x$.
\end{cor}

\begin{theorem} \label{periodic_is_rat}
All $x\in \widetilde{X}$ which have eventually periodic orbits are rational.
\end{theorem}

\begin{proof}
On the interior of an integer interval $f(x) = ax+b$ where $a$ and $b$ are rational. This implies that the $f$ preimage of a rational number is a set of rationals and the image of a rational is rational. This means that it suffices to show that all periodic points are rational.\medskip

Let us suppose that the orbit of $x\in \widetilde{X}$ is $n$ periodic, thus $x = f^n(x)$. Let $\underline{i}(x)=(i_0,i_1...)$. From the definition of $\widetilde{X}$ we have that $f^n(x) \in \interior (J_{i_n})$, thus the restriction of $f^n$ to $\widetilde{J}_{\underline{i}(x)|n}$ is a composition of functions of the form $ax+b$, where $a$ and $b$ are rational, thus $f^n$ itself is a function of the form $ax+b$. Since $|f'(J_k)|\geq 2$ for all $k$, we have that $|a|>1$. From periodicity we have that $x = ax+b$, thus $x = \frac{b}{1-a}$, a rational number.
\end{proof}
For the following proofs recall that for any $a>b$ integers and for any $x_n$ sequence of reals by definition $\prod_{n=a}^b x_n = 1$ and $\sum_{n=a}^b x_n = 0$, as they are empty products and sums.

\begin{theorem} \label{bondolt}
The orbit of all rationals $x\in \widetilde{X}$ for which $\liminf_{k\rightarrow \infty} |f^k(x)| < \infty$ is eventually periodic.
\end{theorem}

\begin{proof}
Let $x = \frac{p}{q}$ and $q$ be positive. Let $\underline{i}(x)=(i_0,i_1...)$. Let $a_k = t(\underline{i}(x)|k)$ be the tag of $\widetilde{J}_{\underline{i}(x)|k}$ as a subset of $\widetilde{J}_{\underline{i}(x)|(k-1)}$, and let $a_0 = i_0 - 1 = L_0(x)$. Let $m_n=|f'(J_{i_n})|=\#\{k: \interior (J_k) \subseteq f(J_{i_n})\}$. By Corollary \ref{regreq} and Theorem \ref{regrDif}: \[
x = L_0(x) + \sum_{k=1}^\infty d_k(x) = \sum_{n=0}^{\infty} \frac{a_n}{\prod_{k=0}^{n-1}m_k}.
\]
We have that $a_0 = \lfloor \frac{p}{q} \rfloor$ and $r_0 = p - qa_0$. By Theorem \ref{regrDif} the following recursion is also true:
\begin{equation} \label{ratrec}
a_{k+1} = \Bigl\lfloor \frac{m_k r_k}{q} \Bigr\rfloor,\quad r_{k+1} = m_k r_k - q a_{k+1}.
\end{equation}

\begin{lemma}
If for some $c,d \in \mathbb{N}$ we have that if $r_c = r_d$ and $i_c = i_d$ and $J_{\underline{i}(x)|c}$ and $J_{\underline{i}(x)|d}$ have the same direction, then $a_{c+1}=a_{d+1},\; r_{c+1}=r_{d+1},\; m_{c+1}=m_{d+1}, \; i_{c+1}=i_{d+1}$ and $J_{\underline{i}(x)|(c+1)}$ and $J_{\underline{i}(x)|(d+1)}$ have the same direction.
\end{lemma}

\begin{proof}
If $i_c = i_d$, then $m_c = m_d$, thus $a_{c+1}=a_{d+1}$ and $r_{c+1}=r_{d+1}$ follow from \eqref{ratrec}. Since $i_c = i_d$ and $a_{c+1}=a_{d+1}$, and we have that $J_{\underline{i}(x)|c}$ and $J_{\underline{i}(x)|d}$ have the same direction, then $i_{c+1}=i_{d+1}$. Hence $J_{\underline{i}(x)|(c+1)}$ and $J_{\underline{i}(x)|(d+1)}$ also have the same direction.
\end{proof}

As a consequence of the previous lemma it follows by induction that if for $c$ and $d$ natural numbers the conditions of the lemma apply, then $i_{c+n} = i_{d+n}$ for any positive integer $n$, thus the orbit is eventually periodic by $|c-d|$. We show that the conditions of the lemma will apply to every rational for some $c$ and $d$, if its orbit has a bounded subsequence.\medskip

If the orbit has a bounded subsequence, then there is some $\alpha \in \mathbb{Z}^+$ such that the orbit has infinitely many elements in $[-\alpha, \alpha]$. In these cases $i_k$ can have only less than $2\alpha$ values. Since $r_k$ is the remainder of an integer modulo $q$ it is an element of the finite set $\{0,...,q-1\}$. There are only two directions an interval can have. By the pigeonhole principle we have that in the first $4\alpha q + 1$ instances, when $f^k(x) \in [-\alpha, \alpha]$ at least once all conditions of the lemma will apply.
\end{proof}

\begin{cor} \label{condpigeon}
If $x\in \tilde{X}, x=\frac{p}{q}$ is not eventually periodic, then for all $n \in \mathbb{Z}^+$ we have that the orbit of $x$ intersects $[-n,n]$ at most $4nq$ times.
\end{cor}

\begin{cor} \label{Ergod}
For all $x\in \tilde{X}, x=\frac{p}{q}$, if the orbit of $x$ is not eventually periodic, then for all $N \in \mathbb{Z}^+$ the following inequality holds:
\begin{equation}
\frac{1}{N} \sum_{n=0}^N |f^n(x)| \geq \frac{N}{8q} - \frac{3}{2} + \frac{4q}{N}.
\end{equation}
\end{cor}

\begin{proof}
Let $S_N = \sum_{n=0}^N |f^n(x)|$. If $x$ is not eventually periodic, then by Corollary \ref{condpigeon} we have that the orbit of $x$ cannot enter $[-n,n]$ more than $4nq$ times. Let us construct an orbit, for which $S_N$ could only be made smaller by breaking the former condition.\medskip

Assume that the orbit visits for all $n \in \{1,...,K\}$ the set $[-n,n] \setminus [-n+1,n-1]$ exactly $4q$ times. This way all the sets of the form $[-n,n]$ are visited the maximum number of times allowed by the condition in Corollary \ref{condpigeon}. With these assumptions the smallest possible sum of these elements is:
$$\sum_{i=0}^{K-1} 4qi = 2qK(K-1).$$
Let us set $K = \bigl\lfloor  \frac{N}{4q} \bigr\rfloor$. Then if the condition of Corollary \ref{condpigeon} is satisfied, then: $$S_N \geq 2qK(K-1) =2q \Bigl\lfloor  \frac{N}{4q} \Bigr\rfloor \Big(\Bigl\lfloor  \frac{N}{4q} \Bigr\rfloor - 1\Big) \geq 2q\Big( \frac{N}{4q} - 1 \Big)  \Big(  \frac{N}{4q}  - 2 \Big).$$
Thus the statement follows:
$$\frac{S_N}{N} \geq \frac{N}{8q} - \frac{3}{2} + \frac{4q}{N}.$$
\end{proof}

\begin{theorem} \label{last}
The orbit of every rational $x\in \widetilde{X}$ is eventually periodic if:
\begin{equation} \label{condi}
0 = \liminf_{N\rightarrow \infty} \max_{n\in \{1-N, ... ,N\}} \lambda \Big(f^{-1} \big( (-\infty, -N] \cup [N, \infty) \big) \cap J_n \Big).
\end{equation}
We refer to \eqref{condi} as the bottleneck condition.
\end{theorem}

\begin{rem}
Let us decipher what does \eqref{condi} tell us. It basically describes that stepping out of the bounded region $[-N, N]$ becomes more and more difficult the bigger $N$ we choose. We know from Theorem \ref{bondolt} that all bounded rational orbits are eventually periodic. We show that with the bottleneck condition satisfied, an unbounded orbit has to approximate the orbit of 1 or 0 so well, that it can only be a rational orbit if it is equal to the orbit of 1 or 0.
\end{rem}

\begin{proof}
Let us consider a rational \[x = \frac{p}{q} = \sum_{n=0}^{\infty} \frac{a_n}{\prod_{k=0}^{n-1}m_k}.\]
By Theorem \ref{bondolt} it is enough to consider the case when $|f^k(x)| \rightarrow \infty$. Let us define the following: \[ \ell_{n_0} := \Big( p - q \sum_{n=0}^{n_0-1} \frac{a_n}{\prod_{k=0}^{n-1}m_k} \Big) \prod_{k=0}^{n_0 - 2} m_k.
\]
By using this notation we can define the following sequence:
\begin{equation} \label{rattalan}
\frac{\ell_{n_0}}{q} = \frac{a_{n_0}}{m_{n_0-1}} + \frac{1}{m_{n_0-1}}\sum_{n=n_0+1}^{\infty} \frac{a_n}{\prod_{k=n_0}^{n-1}m_k}.
\end{equation}

Since $a_n$ is an element of $\{0,...,m_{n-1}-1\}$ if $n>0$, it is true that $\frac{a_n}{m_{n-1}} \in [0,1]$ if $n>0$.

\begin{lemma} \label{position}
Suppose that we have $x \in \widetilde{X}$, then for all $n_0 \in \mathbb{N}$: \[0 < \sum_{n=n_0+1}^{\infty} \frac{a_n}{\prod_{k=n_0}^{n-1}m_k} < 1.\]
\end{lemma}

\begin{proof}
Since $x = L_{n_0}(x) + \sum_{n=n_0+1}^{\infty} \frac{a_n}{\prod_{k=0}^{n-1}m_k}$ we have that:\[
(x - L_{n_0}(x))\prod_{k=0}^{n_0-1}m_k = \sum_{n=n_0+1}^{\infty} \frac{a_n}{\prod_{k=n_0}^{n-1}m_k}.
\]
We also know as a consequence of Lemma \ref{Rendszer3}, that $\prod_{k=0}^{n_0-1} \frac{1}{m_k} = \lambda(J_{\underline{i}(x)|n_0})$. This implies that the function $\phi(y) := (y - L_{n_0}(x))\prod_{k=0}^{n_0-1}m_k$ is the linear map which maps the two endpoints of $\interior (J_{\underline{i}(x)|n_0})$ onto the two endpoints of $(0,1)$ while preserving direction. The value of the sequence at $n_0$ is $\phi(x)$, and if $x \in \widetilde{X}$, then $x \in \interior (J_{\underline{i}(x)|n_0})$, thus $\phi(x) \in (0,1)$.
\end{proof}

\begin{cor} \label{nono}
For any $x \in \widetilde{X}$ we have that $0 < \frac{\ell_{n_0}}{q} < 1$ for all $n_0 \in \mathbb{Z}^+$.
\end{cor}

\begin{proof}
The sequence has the following form:\[
\frac{\ell_{n_0}}{q} = \sum_{n=n_0}^{\infty} \frac{a_n}{\prod_{k=n_0-1}^{n-1}m_k} = (x - L_{n_0-1}(x))\prod_{k=0}^{n_0-2}m_k.
\]
By a change of indices, it follows from Lemma \ref{position} that $0 < \frac{\ell_{n_0}}{q} < 1$. It also follows that $\frac{\ell_{n_0}}{q}$  describes the position of $x$ inside of $\widetilde{J}_{\underline{i}(x)|n_0-1}$.
\end{proof}

Let us assume that the set of numbers which are in the sequence $\frac{\ell_{n_0}}{q}$ has a rational limit point. Let this be $\frac{c}{d}$ for $c \in \mathbb{Z}$ and $d \in \mathbb{Z}^+$. In this case for infinitely many $n_0$:
\[
0 < \Big| \frac{\ell_{n_0}}{q} - \frac{c}{d} \Big| < \frac{1}{qd} .
\]
This is a contradiction, since $| \frac{\ell_{n_0}}{q} - \frac{c}{d} | = \frac{|\ell_{n_0}d - cq|}{qd}$ so if it is non-zero, then it is at least $\frac{1}{qd}$. We will show that if the orbit of $x$ is not eventually an integer, meaning that $x\in \widetilde{X}$, and $x$ is rational, then if $\liminf_{k \rightarrow \infty} |f^k(x)| = \infty$, then the set of numbers in the sequence $\frac{\ell_{n_0}}{q}$ has a rational limit point, and thus we have a contradiction. This will prove the proposition.\medskip

For $x\in \widetilde{X} \cap \mathbb{Q}$ proceeding towards a contradiction we assume that $|f^k(x)|$ tends to infinity. Since the orbit of $x$ is unbounded, we have that for every integer $N > |x|$ there is a non-negative integer $k$ for which $f^k(x)\in [-N, N]$ and $f^{k+1}(x) \notin [-N,N]$.\medskip

For a fixed $N$, for this $k$ let us consider the nonempty sets $\widetilde{J}_{i_k h}$, where $h \in \mathbb{Z}$. Let us define the following set: $$H = \{h \in \mathbb{Z}: f(\widetilde{J}_{i_k h}) \cap ((-\infty, N] \cup [N,\infty)) \neq \emptyset \}.$$ 
We have that the set $t(\{i_kh: h \in H\})$ is either $\{0,...,c\}$ for some $c \in \mathbb{N}$ or it is $\{c,...,m_k-1\}$ for some $c \in \mathbb{N}, c\leq m_k -1$ or it is $\{0,...,c\} \cup \{d,...,m_k-1\}$ for some $c,d \in \mathbb{N}, 0\leq c<d \leq m_k-1$. This is a consequence of that the function $f$ is monotone on $\widetilde{J}_{i_k}$. This implies that for a fixed $N > 0$ we have that: 
\begin{equation} \label{limit1}
\min \Big(\frac{a_{k+1}}{m_{k}}, 1 - \frac{a_{k+1}}{m_{k}} \Big) < \max_{n\in \{1-N, ... ,N\}} \lambda (f^{-1}( (-\infty, -N] \cup [N, \infty)) \cap J_n).
\end{equation}
By \eqref{limit1} the sequence $\frac{a_{n_0}}{m_{n_0-1}}$ has a subsequence which tends to zero or one. If the right hand side in \eqref{limit1} is $\varepsilon > 0$ for $n_0 = k$ this implies that $0 < \frac{1}{m_{n_0}} < \varepsilon$. This combined with Lemma \ref{position} implies $0 \leq \frac{1}{m_{n_0-1}}\sum_{n=n_0+1}^{\infty} \frac{a_n}{\prod_{i=n_0}^{n-1}m_i} < \varepsilon$. By \eqref{rattalan} it follows that $\frac{\ell_{n_0}}{q}$ has a subsequence which tends to zero or one. We also have by Corollary \ref{nono}, that $0 < \frac{\ell_{n_0}}{q} < 1$. Hence the set of numbers in the sequence has zero or one as a limit point, which is rational, thus we have a contradiction. Hence for all $x \in \widetilde{X} \cap \mathbb{Q}$ we have that $\liminf_{k\rightarrow \infty} |f^k(x)| < \infty$, thus by Theorem \ref{bondolt} the statement follows.
\end{proof}

The work \cite{Oppenheim} on Cantor series should be mentioned here, as Theorems 3 and 4 of that paper are implicitly proven and then applied in the previous proof. The theorems are on the rationality of a Cantor series of the following form:
\begin{equation}
x = a_0 + \sum_{i=1}^\infty \frac{a_i}{\prod_{j=1}^i b_i}.
\end{equation}
Where the following are satisfied for all $i \in \mathbb{N}^+$:
\begin{equation} \label{condopp}
0 \leq a_i \leq b_i - 1, \quad b_i \geq 2.
\end{equation}

\begin{theorem}(Theorem 3. of \cite{Oppenheim})
If condition \eqref{condopp} is satisfied, and $a_i < b_i - 1$ infinitely often and there exists a subsequence $a_{i_n}, b_{i_n}$ of $a_n, b_n$ such that $\nicefrac{a_{i_n}}{b_{i_n}} \to 1$ as $n \to \infty$, then $x$ must be irrational.
\end{theorem}

\begin{theorem}(Theorem 4. of \cite{Oppenheim})
If condition \eqref{condopp} is satisfied, and $a_i > 0$ infinitely often and there exists a subsequence $a_{i_n}, b_{i_n}$ of $a_n, b_n$ such that $\nicefrac{a_{i_n}}{b_{i_n}} \to 0$ and $b_{i_n} \to \infty$ as $n \to \infty$, then $x$ must be irrational.
\end{theorem}

\begin{cor}
In the system $(\mathbb{R}_{\geq 0}, T)$ all the eventually periodic points are rational and for all $x \in \mathbb{Q}_{\geq 0}$ the orbit of $x$ is either eventually periodic or eventually reaches an integer.
\end{cor}

\begin{proof}
Since Theorem \ref{periodic_is_rat} applies, we have that all the eventually periodic points are rational. If $N_k = 4^k$, then: $$\max_{n\in \{1-N_k, ... ,N_k\}} \lambda \Big(f^{-1} \big( (-\infty, -N_k] \cup [N_k, \infty) \big) \cap J_n \Big) = \frac{1}{4^k + 1}.$$ Thus Theorem \ref{last} applies.
\end{proof}

\subsection{Example systems} \label{example_systems}

An interesting application of $\ell$-Markov partitions is that one may associate a Markov chain to systems possessing this property. Similar ideas are present in literature as well, we cite Theorem 9.2.1 of \cite{Boyarsky}. We explore these ideas in more detail in our paper in preparation \cite{mineOfficial}. We introduce some necessary definitions.\medskip

Let us consider a time-homogeneous Markov chain on a state space $I$. Let $p_{ij}$ be the transition probabilities of the Markov chain, where $i,j \in I$. 
\begin{defin}
The following are parameters of the Markov chain:
\begin{enumerate}
\item $f_{ij}^{(n)} = \sum_{s \in S} p_{is_1}p_{s_{n-1}j}\prod_{k=2}^{n-1}p_{s_{k-1}s_k}$ where $s = (s_1, ..., s_{n-1})$ and $S$ is the set of $n-1$ long positive integer sequences not containing $j$. This is the probability of starting from $i$ and taking a path arriving at $j$ for the first time after exactly $n$ steps.
\item $f_{ij}^* = \sum_{n=1}^{\infty} f_{ij}^{(n)}$ the probability of eventually getting to $j$ if we start from $i$.
\item If $f_{ii}^* =1$, then we define $m_i = \sum_{n=1}^{\infty} nf_{ii}^{(n)}$, otherwise we define $m_i = \infty$. This is the average return time to $i$ if we start from it.
\end{enumerate}
\end{defin}
\begin{defin}
We call a state of a Markov chain transient, if $f_{ii}^* <1$. We call a state null-recurrent, if $f_{ii}^* = 1$ and $m_i = \infty$ and positive-recurrent if $f_{ii}^* = 1$ and $m_i < \infty$.
\end{defin}
\begin{defin}[Associated Markov chain]
Let $(X,f)$ be a dynamical system with an $\ell$-Markov partition $\mathscr{H}$. The associated Markov chain to this system is a time-homogeneous Markov chain on a state space $I$ with the same cardinality as $\mathscr{H}$, and transition probabilities $$p_{ij} = \frac{\lambda(f^{-1}(H_j) \cap H_i)}{\lambda(H_i)}.$$ 
\end{defin}

At a first glance it is not obvious how does the eventual periodicity of rational points and the bottleneck condition relate to the transience and recurrence properties of the Markov chain described by a given system. From statements like Corollary \ref{Ergod} it seems possible that there might be a connection between the mentioned properties. The following four examples show that there is no particular connection between them.\medskip 

In some cases rationals have an orbit which is not eventually periodic, thus the symbolic trajectory $i_n$ is not eventually periodic, however the sequence $a_n$ of tags describing the Cantor series of the given point is still always eventually periodic. Example \ref{kek} is such a system. To account for this possibility Example \ref{posrecEscapes} and \ref{extransi} are systems, where $a_n = i_n-1$ always for any orbit. The properties of each example given are listed below:

\begin{itemize}
  \item[\ref{zold}]  Theorem \ref{last} applies, however its Markov chain has only transient states.
  \item[\ref{kek}] Has rational orbits which are not eventually periodic and do not hit integers, however every state is positive recurrent.
  \item[\ref{posrecEscapes}] For any orbit $a_n = i_n-1$, has rational orbits which are not eventually periodic and do not hit integers, however every state is positive recurrent.
  \item[\ref{extransi}] For any orbit $a_n = i_n-1$, Theorem \ref{last} applies, however its Markov chain has only transient states.
\end{itemize}

\begin{exam} \label{zold}
We define a function $f: \mathbb{R}_{\geq 0} \rightarrow \mathbb{R}_{\geq 0}$ for which Theorem \ref{last} applies however it has only transient states. Let us define a sequence $\{n_k\}_{k=1}^\infty \in (\mathbb{Z}^+)^\mathbb{N}$, which tends to infinity and let $s_k = \sum_{i=1}^{k-1} n_i$. If $s_k < n \leq s_{k+1}$, then let $f|_{\interior (I_n)}(x) = (n_k+1)x + s_k$. The $f$ images of integers does not matter. For an example of such a system see the green graph on Figure \ref{vonalak}. For the sequence $s_k$:
$$\max_{n\in \{1-s_k, ... ,s_k\}} \lambda \Big(f^{-1} \big( (-\infty, -s_k] \cup [s_k, \infty) \big) \cap J_n \Big) = \frac{1}{n_{k-1}}.$$
Thus Theorem \ref{last} applies, therefore for this system all eventually periodic points are rational and all rationals either hit an integer or are eventually periodic. The Markov chain defined by the system has no recurrent state, since for a state $s_k < n \leq s_{k+1}$ we eventually step above $s_{k+1}$ with probability one, however we cannot step under $s_k$. 
\end{exam}
\begin{exam} \label{kek}
There are positive recurrent Markov chains, which can be described by a function satisfying the assumptions of Definition \ref{z1} and still do not satisfy the bottleneck condition. Take the random walk on $\mathbb{Z}^+$ where from $n \geq 3$ we step on $n+1, n, n-1$ or $n-2$ with the probability of each being $\nicefrac{1}{4}$, and from one or two we step onto an element of $\{1,2,3,4\}$ with the probability of each being $\nicefrac{1}{4}$. Take a function which defines this Markov chain.
\begin{equation}
f(x) =
   \begin{cases} 
      4(x-\lfloor x \rfloor) & \text{if } 0 \leq x < 3, \\
      4x - 3\lfloor x \rfloor -2 & \text{if } 3 \leq x.
   \end{cases}
\end{equation}
For a plot of this function see Figure \ref{vonalak}. The set of integer intervals $\{J_k \}_{k=1}^{\infty}$ forms an expansive $\ell$-Markov partition of the domain of this function. The bottleneck condition does not apply since: $$\max_{n\in \{1-N, ... ,N\}} \lambda \Big(f^{-1} \big( (-\infty, -N] \cup [N, \infty) \big) \cap J_n \Big) \geq \lambda \Big(f^{-1} \big( (-\infty, -N] \cup [N, \infty) \big) \cap J_N \Big) \geq \frac{1}{4}.$$
\end{exam}

\begin{figure}[h]
\includegraphics[width= 14cm]{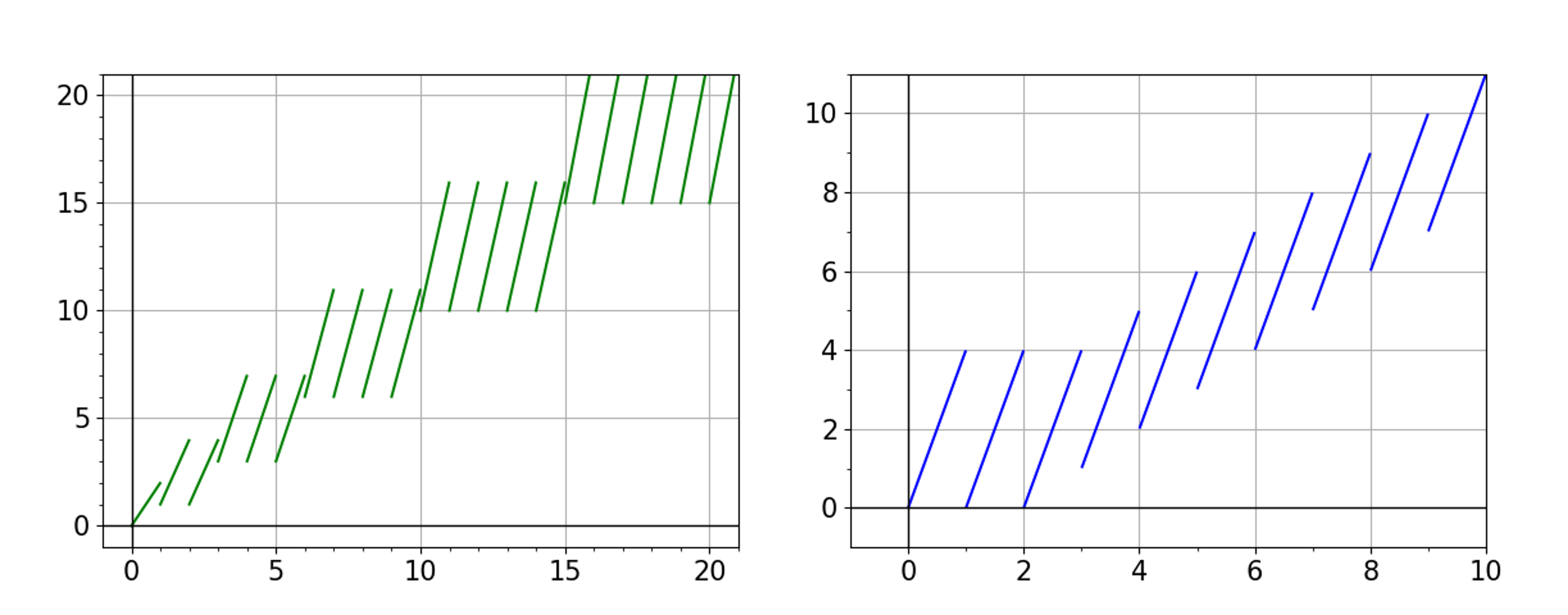}
\centering
\caption{Plots of the example systems. The green one to the left is for Example \ref{zold}, the blue one to the right is for Example \ref{kek}.}
\label{vonalak}
\end{figure}

Example \ref{kek} is a system, which has not eventually periodic rational points which do not hit integers. An example is the number $x = 0.\overline{32}$ in base-4, which is $\nicefrac{14}{15}$. Even though the orbit of this point tends to infinity, the sequence of $a_n$-s present in the Cantor series representation is eventually periodic. Namely it is $0,3,2,3,2...$. To avoid cases where the eventual periodicity of the $a_n$ and $i_n$ sequences differ we can consider systems, where $X = \mathbb{R}_{\geq 0}$ and $f|_{J_n} = kx - k(n-1)$ for some integer $k \geq 2$ and $J_n = [n-1,n]$. In this case on all integer intervals the graph of $f$ is a line with a positive slope, intersecting the $x$ axis at the beginning of the interval. This way for all $x$, we have that $a_n = i_n - 1$ for all $n \in \mathbb{N}$, where the symbolic orbit of $x$ is $(i_0, i_1, ...)$. In this case the eventual periodicity of the $a_n$ and $i_n$ sequences will coincide.

\begin{exam} \label{posrecEscapes}
We show a system $(X,f)$, for which the associated Markov chain is positive recurrent and $\nicefrac{1}{2}$ has an orbit which is not eventually periodic and does not hit integers. In addition the system will also satisfy that $a_n = i_n - 1$ for all $x \in \tilde{X}$. The domain is $\mathbb{R}_{\geq 0}$, and we call our transformation $f$. Almost all integer intervals will be mapped onto $[0,3]$, however the orbit of $\nicefrac{1}{2}$ will never return to $[0,3]$. We refer to Figure \ref{ketted}.\medskip

Let us consider the sequence $S_n$ for which $S_0 = 1, S_1= 3$ and for all $n \geq 2$ we have that $S_{n+1} = 4S_n + 1$. We define a sequence $h_n$ and we will define $f$ such that $f'(J_n) = h_n$. Let $h_{\lfloor \nicefrac{S_n}{2} \rfloor + 1} = S_{n+1}$. For all other $n$ let $h_n=3$. In this case we will have that $f^{k}(\nicefrac{1}{2}) = \nicefrac{S_{k}}{2}$. The associated Markov-chain with this system will be positive recurrent. Let $P = (p_{ij})_{i,j \in \mathbb{N}^+}$ be the transition matrix. We shall see that from any state we return to the first, second or third state in finite expected time. As the associated Markov chain is irreducible and aperiodic, this implies the positive recurrence of every state.\medskip

Let $X_0^{(n)}, X_1^{(n)}, X_2^{(n)} ... : \Omega \to \mathbb{N}^+$ be random variables realising the Markov chain given that $\mathbb{P}(X_0^{(n)} = n) = 1$. Let us define the expected time that we return to $\{1,2,3\}$ given that we started from $n$. For this end we first define the random variable $Y^{(n)}: \Omega \to \mathbb{N}^+$:
$$Y^{(n)}(\omega) = \inf \{k\in \mathbb{N}^+: X_k^{(n)}(\omega) \in \{1,2,3\}\},$$
$$E_n = \mathbb{E}(Y^{(n)}) = \mathbb{E}[k\in \mathbb{N}^+: X_k^{(n)} \in \{1,2,3\} \land X_i^{(n)} \notin \{1,2,3\}\quad  \forall i \in \{1, ..., k-1\}].$$
If $h_n = 3$, then $E_n = 1$ as $p_{n1} + p_{n2} + p_{n3} = 1$. For any $n$ we have that:
\begin{equation} \label{usedlabel}
\sum_{\substack{j \in \mathbb{N}^+ \\ h_j \neq 3}} p_{nj} \leq \frac{1}{3}.
\end{equation}
Now we can estimate $E_n$. From now on we assume that $h_n \neq 3$.
$$E_n = \sum_{k=0}^\infty \mathbb{P}(Y^{(n)} > k).$$
The event $\{\omega \in \Omega: Y^{(n)}(\omega) > k\}$ is equivalent to the event $\{\omega \in \Omega: X_1^{(n)}(\omega), ..., X_k^{(n)}(\omega) \notin \{1,2,3\}\}$. We have the following inclusion:
\begin{equation} \label{d1}
\{\omega \in \Omega: X_1^{(n)}(\omega), ..., X_k^{(n)}(\omega) \notin \{1,2,3\}\} \subseteq \{\omega \in \Omega: h_{X_0^{(n)}(\omega)}, ..., h_{X_{k-1}^{(n)}(\omega)} \neq 3\}.
\end{equation}
Let us denote the event $\{\omega \in \Omega: h_{X_0^{(n)}(\omega)}, ..., h_{X_{k-1}^{(n)}(\omega)} \neq 3\}$ by $A_k^{(n)}$. By \eqref{d1} we obtain that $\mathbb{P}(Y^{(n)} > k) \leq \mathbb{P}(A_k^{(n)})$. We deduce the following:
\begin{equation} \label{orbitd1}
\mathbb{P}(A_k^{(n)}) = \sum_{\substack{a_1, ..., a_{k-1} \in \mathbb{N}^+ \\ h_{a_1}, ..., h_{a_{k-1}} \neq 3}} p_{n a_1} p_{a_1 a_2} ... p_{a_{k-2} a_{k-1}}.
\end{equation}
By rearranging \eqref{orbitd1} we get the following: 
\begin{equation} \label{orbitd12}
\mathbb{P}(A_k^{(n)}) = \sum_{\substack{a_1 \in \mathbb{N}^+ \\ h_{a_1} \neq 3}} \Big( p_{na_1} \sum_{\substack{a_2 \in \mathbb{N}^+ \\ h_{a_2} \neq 3}} \Big( p_{a_1 a_2} ... \sum_{\substack{a_{k-1} \in \mathbb{N}^+ \\ h_{a_{k-1}} \neq 3}} p_{a_{k-2}a_{k-1}} \Big) ... \Big).
\end{equation}
By \eqref{usedlabel} and \eqref{orbitd12} we can conclude that $\mathbb{P}(A_k^{n}) \leq 3^{1-k}$. Finally we have the following upper bounds for $E_n$ if $h_n \neq 3$:
\begin{equation}
E_n = \sum_{k=0}^\infty \mathbb{P}(Y^{(n)} > k) \leq \sum_{k=0}^{\infty} \mathbb{P}(A_k^{(n)}) \leq 3 \sum_{k=0}^{\infty} 3^{-k} =  4.5 .
\end{equation}
Recall that if $h_n = 3$, then $E_n =1$. Hence for all $n$ we have that $E_n \leq 4.5$. Since this value is finite, the Markov-chain associated with the original system must be positive recurrent.
\end{exam}

\begin{figure}[h]
\includegraphics[width= 9.5cm]{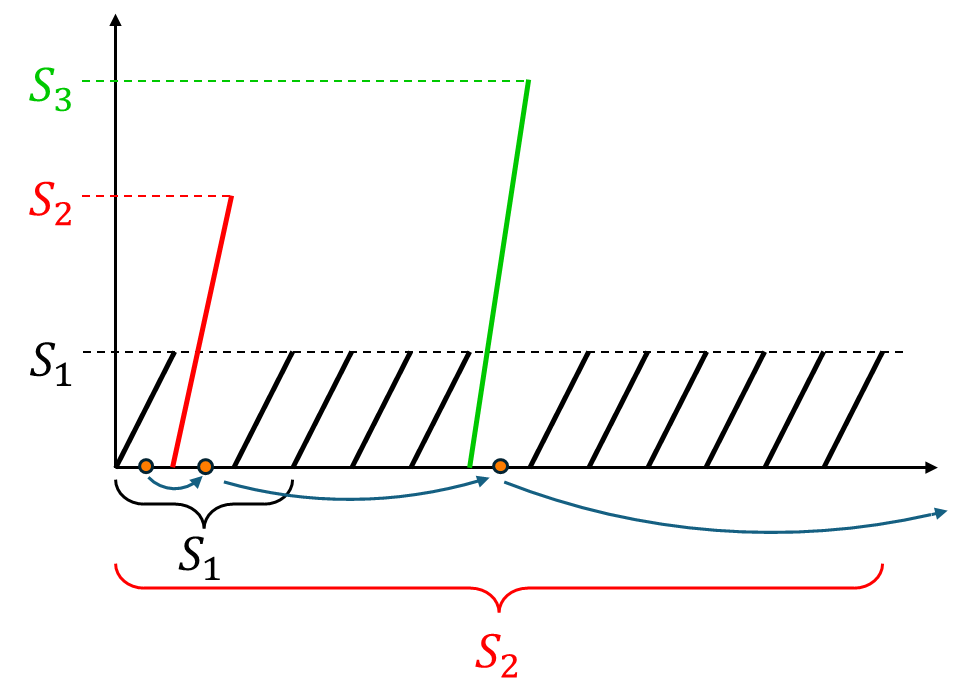}
\centering
\caption{Illustration of the systems appearing in Example \ref{posrecEscapes}, with the orbit of $\nicefrac{1}{2}$ highlighted by the orange dots. The $y$ axis is not to scale.}
\label{ketted}
\end{figure}

\begin{exam} \label{extransi}
We show an example of a system $(\mathbb{R}_{\geq 0}, f)$, such that Theorem \ref{last} applies and the chain is transient. We now give a heuristic explanation of what we plan to do. We take inspiration from Example \ref{zold}. Consider a Markov chain for which from all states the probability of stepping to a larger state is $\nicefrac{1}{n}$. We also assume that the conditional probability of stepping to a larger state given that we do not stay where we are approaches $1$ extremely quickly. Let this conditional probability for the $n$th state be $p(n)$. This will make the Markov chain transient. We will represent one state of this Markov chain by lots of states of our system. We refer to Figure  \ref{transi}. We assume that $p(n)$ is monotonically increasing.\medskip 

Let the sequence $p(n)$ be given. For example let $p(n) = e^{-n^{-2}}$ as in Lemma \ref{harang}. Let us define $P(n) = \lceil \frac{p(n)}{1-p(n)} \rceil$. Let $S_0 = 0, S_1 = 1, h_1 = 2$ and for all $n>1$ let $S_n = (n-1)P(n)S_{n-1}$ and $h_n = nP(n)S_{n-1}$. Let $L_n = [S_{n-1}, S_n]$. Let us consider an $(\mathbb{R}_{\geq 0}, f)$ system for which if $J_n \subset L_n$, then $f|_{J_n}(x) = h_n x - h_n (n-1)$. For this system Theorem \ref{last} applies, since:
$$\max_{n\in \{1-S_N, ... ,S_N\}} \lambda \Big(f^{-1} \big( (-\infty, -S_N] \cup [S_N, \infty) \big) \cap J_n \Big) = \frac{h_N - S_N}{h_n}= \frac{1}{N} \to 0.$$

\begin{figure}[h]
\includegraphics[width= 10cm]{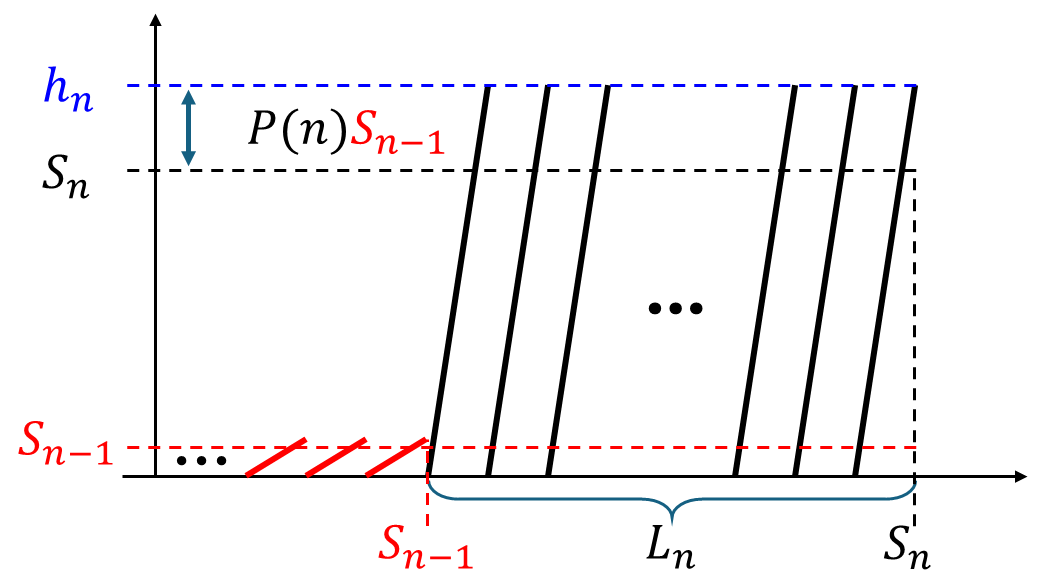}
\centering
\caption{Illustration of the systems appearing in Example \ref{extransi}.}
\label{transi}
\end{figure}

\begin{lemma}\label{harang}
If $p(n) = e^{-n^{-2}}$, then the Markov chain associated with the resulting system will only have transient states.
\end{lemma}
\begin{proof}
Let us consider a sequence of random variables $X_0, X_1, X_2, ... : \Omega \to \mathbb{Z}^+$ representing a random walk on the Markov chain defined by the system, where $\Omega$ is the sample space. Let $\mathcal{L}_n := \{i\in \mathbb{N}: J_i \subset L_n\}$. We have the following: 
$$\mathbb{P}[X_{k+1} \in \mathcal{L}_{n+1} | X_{k+1} \notin \mathcal{L}_{n}, X_k \in \mathcal{L}_n] = \frac{P(n)S_{n-1}}{P(n)S_{n-1} + S_{n-1}} \geq p(n).$$
We also have that $X_{j}(\omega)$ eventually leaves $\mathcal{L}_n$ with probability one. Let $n_j(\omega)$ be the integer for which $X_{j}(\omega) \in \mathcal{L}_{n_j}$. With this the probability of the sequence $n_j(\omega)$ never decreasing is at least $\prod_{k=n}^{\infty} p(n)$. If $\prod_{k=1}^{\infty} p(n) > 0$, then $X_j$ will tend to infinity almost surely. We can choose $p(n) = e^{-n^{-2}}$, thus $\prod_{k=1}^{\infty} p(n) = e^{-\frac{\pi^2}{6}} > 0$.
\end{proof}
\end{exam}

\addcontentsline{toc}{section}{Acknowledgement}
\section*{Acknowledgement}
I would like to express my deepest gratitude to my advisor Zoltán Buczolich, for proposing these intriguing problems and providing guidance throughout the preparation of this paper.

\bibliographystyle{plain}
\addcontentsline{toc}{section}{References}
\bibliography{hmm2}

\end{document}